%% file: main.tex
\documentclass[10pt,english,a4paper]{article}

\usepackage[utf8]{inputenc}
\usepackage[T1]{fontenc}
\usepackage[english]{babel}
\usepackage[charter]{mathdesign}
\usepackage[mathscr]{euscript}
\usepackage{amsmath,amsthm}

\usepackage{enumitem}
\usepackage[explicit]{titlesec}
\usepackage{twoopt}
\usepackage{hyperref}
\usepackage{microtype} 

\usepackage{tikz,tikz-cd}
\usetikzlibrary{arrows}
\usetikzlibrary{matrix}

\hypersetup{
    colorlinks,
    linkcolor={red!50!black},
    citecolor={blue!50!black},
    urlcolor={blue!60!black}
}

\input{macros.tex} 
\newcommand*{\weakcons}{hCon\relax}
\newcommand*{\hbc}{hBC\relax}
\newcommand*{\quillenpushpull}{Q\relax}
\newcommand*{\totalcof}[1]{\class C_{#1}}
\newcommand*{\totalfib}[1]{\class F_{#1}}
\newcommand*{\totalweak}[1]{\class W_{#1}}
\newcommand*{\totalacof}[1]{\totalcof{#1}^{\sim}}
\newcommand*{\totalafib}[1]{\totalfib{#1}^{\sim}}

\newcommand{\Bcategory}{\cat{B}}

\newcommand{\Ccategory}{\cat{C}}

\newcommand{\Ecategory}{\cat{E}}

\newcommand{\Mcategory}{\cat{M}}
\newcommand{\Ncategory}{\cat{N}}
\newcommand{\Ppseudofunctor}{\cat{P}}

\newcommand{\Quillen}{\concrete{Quil}}

\title{\sffamily On bifibrations of model categories}
\author{\normalsize%
  \begin{minipage}{.48\linewidth}\centering
    Pierre Cagne\\
    IRIF, Universit\'e Paris 7\\
    \href{mailto:cagne@irif.fr}%
    {\nolinkurl{cagne@irif.fr}}%
  \end{minipage}
  \begin{minipage}{.48\linewidth}\centering
    Paul-Andr\'e Melli\`es\\
    IRIF, Universit\'e Paris 7\\
    \href{mailto:mellies@irif.fr}%
    {\nolinkurl{mellies@irif.fr}}%
  \end{minipage} 
}
\date{\normalsize Last updated on \today}

\begin{document}

\maketitle

\begin{abstract}
  In this article, we develop a notion of Quillen bifibration
  which combines the two notions of
  Grothendieck bifibration and of Quillen model structure.
  In particular, given a bifibration $p:\cat E\to\cat B$, 
  we describe when a family of model structures
  on the fibers $\cat E_A$ and on the basis category $\cat B$
  combines into a model structure on the total category~$\cat E$, such
  that the functor~$p$ preserves cofibrations, fibrations and weak equivalences.
  Using this Grothendieck construction for model structures,
  we revisit the traditional definition of Reedy model structures,
   and possible generalizations, and exhibit their bifibrational nature.
\end{abstract}

\tableofcontents

\section{Introduction}

In this paper, we investigate how the two notions of
\emph{Grothendieck bifibration} and of \emph{Quillen model category}
may be suitably combined together.
We are specifically interested in the situation of a Grothendieck
bifibration $p\from  \Ecategory\to\Bcategory$ where the basis
category~$\Bcategory$ as well as each fiber $\Ecategory_A$ for an
object~$A$ of the basis category~$\Bcategory$ is equipped with a
Quillen model structure.
Our main purpose will be to identify necessary and sufficient
conditions on the Grothendieck bifibration
$p\from  \Ecategory\to\Bcategory$ to ensure that the total category
$\Ecategory$ inherits a model structure from the model structures
assigned to the basis~$\Bcategory$ and to the fibers~$\Ecategory_A$'s.
We start our inquiry by recalling the fundamental relationship between
bifibrations and adjunctions.
This connection will guide us all along the paper.
Our plan is indeed to proceed by analogy,
and to carve out a notion of \emph{Quillen bifibration}
playing the same role for Grothendieck bifibrations
as the notion of \emph{Quillen adjunction}
plays today for the notion of adjunction.

\medbreak

\paragraph{Grothendieck bifibrations and adjunctions.}
We will generally work with \emph{cloven} bifibrations.
Recall that a \emph{cleavage on a Grothendieck fibration} is a choice,
for every morphism~$u\from A\to B$ and for every object~$Y$ above~$B$,
of a cartesian morphism $\cart u Y\from \pull u Y\to Y$
above~$u$.
Dually, a \emph{cleavage on a Grothendieck opfibration} is a choice,
for every morphism~$u\from A\to B$ and for every object~$X$ above~$A$,
of a left cartesian morphism $\cocart u X \from X\to \push u X$
above~$u$.
%
%
In a cloven Grothendieck fibration, every morphism $u\from A\to B$ in
the basis category~$\Bcategory$ induces a functor
\begin{equation}
  \label{eq:right-cleavage}%
  \begin{tikzcd}[column sep=.6em]
    u^{\ast} & : & \Ecategory_B\arrow[rrrr] &&&& \Ecategory_A
  \end{tikzcd}
\end{equation}
Symmetrically, in a cloven Grothendieck opfibration, every morphism
$u\from A\to B$ in the basis category~$\Bcategory$ induces a functor
\begin{equation}
  \label{eq:left-cleavage}%
  \begin{tikzcd}[column sep=.6em]
    u_{!} & : & \Ecategory_A\arrow[rrrr] &&&& \Ecategory_B
  \end{tikzcd}
\end{equation}
A \emph{cloven bifibration} (or more simply a bifibration) is a left
and right Grothendieck fibration~$p\from \Ecategory\to\Bcategory$
equipped with a cleavage on both sides.

\medbreak

Formulated in this way, a bifibration $p\from \Ecategory\to\Bcategory$
is simply the ``juxtaposition'' of a left and of a right Grothendieck
fibration, with no apparent connection between the two structures.
Hence, a remarkable phenomenon is that the two fibrational structures
are in fact strongly interdependent.
Indeed, it appears that in a bifibration
$p\from \Ecategory\to\Bcategory$, the pair of functors
(\ref{eq:right-cleavage}) and (\ref{eq:left-cleavage}) associated to a
morphism $u\from A\to B$ defines an adjunction between the fiber
categories
\begin{displaymath}
  \begin{tikzcd}[column sep=.8em]
    u_{!} & : & {\Ecategory_A} \arrow[rrrr, yshift=-.5ex] 
    &&&&
    {\Ecategory_B}  \arrow[llll, yshift=.5ex] 
    & : & u^{\ast}
  \end{tikzcd}
\end{displaymath}
where the functor $u_{!}$ is left adjoint to the functor~$\pull u$.
%
%
%
The bond between bifibrations and adjunctions
is even tighter when one looks at it from
the point of view of indexed categories.
%
Recall that a (covariantly) \emph{indexed category} of basis~$\Bcategory$
is defined as a pseudofunctor
\begin{equation}
  \label{eq:pseudofunctor}%
  \begin{tikzcd}[column sep=1em]
    \Ppseudofunctor & : & \Bcategory\arrow[rrrr] &&&& \Cat
  \end{tikzcd}
\end{equation}
where $\Cat$ denotes the 2-category of categories, functors and
natural transformations.
Every cloven Grothendieck opfibration $p\from \Ecategory\to\Bcategory$ induces
an indexed category $\Ppseudofunctor$ which transports every
object~$A$ of the basis~$\Bcategory$ to the fiber category
$\Ecategory_A$, and every morphism $u\from A\to B$ of the basis to the
functor $\push u \from \Ecategory_A\to\Ecategory_B$.
Conversely, the Grothendieck construction enables one to construct a cloven Grothendieck opfibration 
$p\from \Ecategory\to\Bcategory$ from an indexed category $\Ppseudofunctor$.
This back-and-forth translation defines an equivalence of categories
between
\begin{center}
\begin{tabular}{ccc}
\begin{tabular}{c}
the category of 
\\
cloven Grothendieck opfibrations
\\
with basis category~$\Bcategory$
\end{tabular}
& \hspace{1em} $\rightleftharpoons$ \hspace{1.3em} &
\begin{tabular}{c}
the category of 
\\
indexed categories
\\
with basis category~$\Bcategory$
\end{tabular}
\end{tabular}
\end{center}
All this is well-known.
What is a little bit less familiar (possibly) and which matters to us
here is that this correspondence may be adapted to Grothendieck
bifibrations, in the following way.
Consider the 2-category $\Adj$ with categories as objects, with
adjunctions
\begin{equation}
  \label{eq:M-N-adjunction}%
\begin{tikzcd}[column sep=.8em]
L & \from  & {\Mcategory} \arrow[rrrr, yshift=-.5ex] 
&&&&
{\Ncategory}  \arrow[llll, yshift=.5ex] 
& \from  & R
\end{tikzcd}
\end{equation}
as morphisms from $\Mcategory$ to $\Ncategory$,
and with natural transformations
\begin{displaymath}
\begin{tikzcd}[column sep=.8em, row sep=0em]
\theta & \from  & 
{L_1} \arrow[rrrr, double, -implies]
&&&&
{L_2}
& \from  & 
\Mcategory \arrow[rrrr]
&&&&
\Ncategory
\end{tikzcd}
\end{displaymath}
between the left adjoint functors as 2-dimensional cells 
$\theta\from (L_1,R_1)\Rightarrow (L_2,R_2)$.
%
In the same way as we have done earlier,
an \emph{indexed category-with-adjunctions}
with basis category $\Bcategory$ is defined as a pseudofunctor
\begin{equation}
  \label{eq:pseudofunctor-to-adj}%
\begin{tikzcd}[column sep=1em]
\Ppseudofunctor & \from  & \Bcategory\arrow[rrrr] &&&& \Adj
\end{tikzcd}
\end{equation}
For the same reasons as in the case of Grothendieck opfibrations,
there is an equivalence of category between
\begin{center}
\begin{tabular}{ccc}
\begin{tabular}{c}
the category of 
\\
cloven bifibrations
\\
with basis category~$\Bcategory$
\end{tabular}
& \hspace{.5em} $\rightleftharpoons$ \hspace{.8em} &
\begin{tabular}{c}
the category of 
\\
indexed categories-with-adjunctions
\\
with basis category~$\Bcategory$
\end{tabular}
\end{tabular}
\end{center}
From this follows, among other consequences, that a cloven bifibration
$p\from \Ecategory\to\Bcategory$ is the same thing as a cloven right
fibration where the functor
$\pull u \from \Ecategory_B\to\Ecategory_A$ comes equipped with a left
adjoint $\push u \from \Ecategory_A\to\Ecategory_B$ for every
morphism~$u\from A\to B$ of the basis category~$\Bcategory$.

\medbreak

By way of illustration, consider the ordinal category
$\lincat 2$ with two objects $0$ and $1$
and a unique non-identity morphism $u\from 0\to 1$.
By the discussion above, a Grothendieck bifibration $p:\Ecategory\to \Bcategory$ 
on the basis category $\Bcategory=\lincat 2$ is the same thing 
as an adjunction~(\ref{eq:M-N-adjunction}).
The correspondence relies on the observation that every
adjunction~(\ref{eq:M-N-adjunction}) can be turned into a bifibration
$p\from \Ecategory\to \Bcategory$ where the category $\Ecategory$ is
defined as the category of \emph{collage} associated to the
adjunction~$(L,R)$, with fibers $\Ecategory_0=\Mcategory$,
$\Ecategory_1=\Ncategory$ and mediating functors $u^*=R$ and $u_!=L$,
see \cite{street:fib-in-bicat} for the notion of collage.
For that reason, the Grothendieck construction for bifibrations may be
seen as a generalized and fibrational notion of collage.


\paragraph{Model structures and Quillen adjunctions.}
Seen from that angle, the notion of Grothendieck bifibration provides
a fibrational counterpart (and also a far-reaching generalization) of
the fundamental notion of adjunction between categories.
This perspective opens a firm connection with modern homotopy theory,
thanks to the notion of \emph{Quillen adjunction} between model categories.
Recall that a \emph{model structure} on a category~$\Mcategory$
delineates three classes $\class C$, $\class W$, $\class F$ of maps
called \emph{cofibrations}, \emph{weak equivalences} and
\emph{fibrations} respectively ; these classes of maps are moreover
required to satisfy a number of properties recalled in
definition~\ref{def:model-category}.
A fibration or a cofibration which is at the same time a weak
equivalence is called \emph{acyclic}.

\begin{remark}
  \label{rem:about-limits-in-model-categories}%
  By extension, we find sometimes convenient to call \emph{model
    structure} a category~$\Mcategory$ \emph{together} with its model
  structure $(\class C, \class W, \class F)$.
  The appropriate name for that notion would be \emph{model category}
  but the terminology is already used in the literature for a
  \emph{finitely complete} and \emph{finitely cocomplete}
  category~$\Ccategory$ equipped with a model structure
  $(\class C, \class W, \class F)$.
  The extra completeness assumptions play a role in the construction
  of the homotopy category $\Ho{\Ccategory}$, and it is thus
  integrated in the accepted definition of a ``model category''.
  We prefer to work with ``model structures'' for two reasons.
  On the one hand, the construction of $\Ho{\Ccategory}$ can be
  performed using the weaker assumption that the category~$\Ccategory$
  has finite products and finite coproducts, as noticed by
  Egger~\cite{egger:model-cat-no-eq}.
  On the other hand, the extra completeness assumptions are
  independent of the relationship between Grothendieck bifibrations
  and model structures, and may be treated separately.
\end{remark}


We recall below the notions of left and right Quillen functor between
model structures.
\begin{definition}[Quillen functors]
  A functor $F\from \Mcategory\to\Ncategory$ between two model
  structures $\Mcategory$ and $\Ncategory$ is called a \emph{left}
  \emph{Quillen} \emph{functor} when it transports every cofibration
  of~$\Mcategory$ to a cofibration of~$\Ncategory$ and every acyclic
  cofibration of~$\Mcategory$ to an acyclic cofibration
  of~$\Ncategory$.
Dually, a functor $F\from \Mcategory\to\Ncategory$ is called a
\emph{right Quillen functor} when it transports every fibration
of~$\Mcategory$ to a fibration of~$\Ncategory$ and every acyclic
fibration of~$\Mcategory$ to an acyclic fibration of~$\Ncategory$.
A functor $F\from \Mcategory\to\Ncategory$ which is at the same time a
left and a right Quillen functor is called a \emph{Quillen functor}.
\end{definition}
A simple argument shows that a Quillen functor
$F\from \Mcategory\to\Ncategory$ transports every weak equivalence
of~$\Mcategory$ to a weak equivalence of~$\Ncategory$.
For that reason, a Quillen functor is the same thing as a functor
which transports every cofibration, weak equivalence or fibration
$f\from A\to B$ of~$\Mcategory$ to a map $Ff\from FA\to FB$ with the
same status in the model structure of~$\Ncategory$.

\medbreak
The notion of \emph{Quillen adjunction} relies on the following
observation.
\begin{proposition}\label{proposition/LR}
Suppose given an adjunction
\begin{equation}
  \label{eq:M-N-Quillen-adjunction}%
  \begin{tikzcd}[column sep=.8em]
    L & \from  & {\Mcategory} \arrow[rrrr, yshift=-.5ex] 
    &&&&
    {\Ncategory}  \arrow[llll, yshift=.5ex] 
    & \from  & R
  \end{tikzcd}
\end{equation}
between two model categories~$\Mcategory$ and~$\Ncategory$.
The following assertions are equivalent:
\begin{itemize}
\item the left adjoint functor~$L\from \Mcategory\to\Ncategory$ is a
  left Quillen functor,
\item the right adjoint functor~$R\from \Ncategory\to\Mcategory$ is a
  right Quillen functor.
\end{itemize}
\end{proposition}
\begin{definition}[Quillen adjunctions]
  An adjunction $L \from \cat M \leftrightarrows \cat N \cofrom R$
  between two model categories~$\Mcategory$ and~$\Ncategory$ is called
  a \emph{Quillen adjunction} when the equivalent assertions of
  Prop.~\ref{proposition/LR} hold.
\end{definition}

\paragraph{Quillen bifibrations.}
At this stage, we are ready to introduce the notion of \emph{Quillen
  bifibration} which we will study in the paper.
We start by observing that whenever
the total category~$\Ecategory$ of a functor
$p\from \Ecategory\to\Bcategory$ is equipped with a model structure
$(\class C_{\Ecategory},\class W_{\Ecategory},\class F_{\Ecategory})$,
every fiber~$\Ecategory_A$ associated to an object~$A$ of the basis
category~$\Bcategory$ comes equipped with three classes of maps noted
$\class C_A$, $\class W_A$, $\class F_A$ called cofibrations, weak
equivalences and fibrations above the object~$A$, respectively.
The classes are defined in the expected way:
\begin{displaymath}
  \class C_A = \class C_{\Ecategory} \cap \class{Hom}_A
  \quad\quad
  \class W_A = \class W_{\Ecategory} \cap \class{Hom}_A
  \quad\quad
  \class F_A = \class F_{\Ecategory} \cap \class{Hom}_A
\end{displaymath}
where $\class{Hom}_A$ denotes the class of maps~$f$ of the
category~$\Ecategory$ above the object~$A$, that is, such that
$p(f)=\id A$.
We declare that the model structure
$(\class C_{\Ecategory},\class W_{\Ecategory},\class F_{\Ecategory})$
on the total category~$\Ecategory$ \emph{restricts} to a model
structure on the fiber~$\Ecategory_A$ when the three classes
$\class C_A$, $\class W_A$, $\class F_A$
satisfy the properties required of a model structure
on the category~$\Ecategory_A$.

%

\medbreak
\noindent
This leads us to the main concept of the paper:
\begin{definition}[Quillen bifibrations]
  A Quillen bifibration $p\from \Ecategory\to\Bcategory$ is a
  Grothendieck bifibration where the basis category~$\Bcategory$ and
  the total category~$\Ecategory$ are equipped with a model structure,
  in such a way that
\begin{itemize}
\item the functor $p\from \Ecategory\to\Bcategory$ is a Quillen functor,
\item the model structure of~$\Ecategory$ restricts to a model
  structure on the fiber~$\Ecategory_A$, for every object~$A$ of the
  basis category~$\Bcategory$.
\end{itemize}
%
\end{definition}
This definition of Quillen bifibration deserves to be commented.
The first requirement that $p\from \Ecategory\to\Bcategory$ is a
Quillen functor means that every cofibration, weak equivalence and
fibration $f\from X\to Y$ of the total category~$\Ecategory$ lies
above a map $u\from A\to B$ of the same status in the model
category~$\Bcategory$.
This condition makes sense, and we will see in
section~\ref{sec:quillen-bifib} that it is satisfied in a number of
important examples.
The second requirement means that the model structure
$(\class C_{\Ecategory},\class W_{\Ecategory},\class F_{\Ecategory})$
combines into a single model structure on the total
category~$\Ecategory$ the family of model structures
$(\class C_A,\class W_A,\class F_A)$ on the fiber
categories~$\Ecategory_A$.
%
%

\paragraph{A Grothendieck construction for Quillen bifibrations.}
The notion of Quillen bifibration is tightly connected to the notion
of Quillen adjunction, thanks to the following observation established
in section~\ref{sec:quillen-bifib}.
\begin{proposition}
  In a Quillen bifibration $p\from \Ecategory\to\Bcategory$, the
  adjunction
  \begin{displaymath}
    \begin{tikzcd}[column sep=.8em]
      \push u & \from  & {\Ecategory_A} \arrow[rrrr, yshift=-.5ex] 
      &&&&
      {\Ecategory_B}  \arrow[llll, yshift=.5ex] 
      & \from  & u^*
    \end{tikzcd}
  \end{displaymath}
  is a Quillen adjunction, for every morphism $u\from A\to B$ of the
  basis category~$\Bcategory$.
\end{proposition}
\noindent
From this follows that a Quillen bifibration induces an \emph{indexed
  model structure}
\begin{equation}
  \label{eq:pseudofunctor-to-quillen}%
  \begin{tikzcd}[column sep=1em]
    \Ppseudofunctor & \from  & \Bcategory\arrow[rrrr] &&&& \Quillen
  \end{tikzcd}
\end{equation}
defined as a pseudofunctor from a model structure~$\Bcategory$ to the
2-category $\Quillen$ of model structures, Quillen adjunctions, and
natural transformations.
Our main contribution in this paper is to formulate necessary and
sufficient conditions for a Grothendieck construction to hold in this
situation.
More specifically, we resolve the following problem.


\medbreak
\noindent
\emph{A. Hypothesis of the problem.}  We suppose given an indexed
Quillen category as we have just defined
in~(\ref{eq:pseudofunctor-to-quillen}) or equivalently, a Grothendieck
bifibration $p\from \Ecategory\to\Bcategory$ where
\begin{itemize}
\item the basis category~$\Bcategory$ is equipped with a model
  structure~$(\class C,\class W,\class F)$,
\item every fiber~$\Ecategory_A$ is equipped with a model structure
$(\class C_A,\class W_A,\class F_A)$,
\item the adjunction $(\push u,\pull u)$ is a Quillen adjunction, for
  every morphism $u\from A\to B$ of the basis category~$\Bcategory$.
\end{itemize}
\noindent
\emph{B. Resolution of the problem.}  We find necessary and sufficient
conditions to ensure that there exists a model structure
$(\class C_{\Ecategory},\class W_{\Ecategory},\class F_{\Ecategory})$
on the total category~$\Ecategory$ such that
\begin{itemize}
\item the Grothendieck bifibration $p\from \Ecategory\to\Bcategory$
  defines a Quillen bifibration,
\item for every object~$A$ of the basis category, the model structure
  $(\class C,\class W,\class F)$ of the total category~$\Ecategory$
  restricts to the model structure
  $(\class C_A,\class W_A,\class F_A)$ of the fiber~$\Ecategory_A$.
\end{itemize}
%
%
We establish in the course of the paper (see
section~\ref{sec:quillen-bifib}) that there exists at most one
solution to the problem, which is obtained by defining the
cofibrations and fibrations of the total category~$\Ecategory$ in the
following way:
\begin{itemize}
\item a morphism $f\from X\to Y$ of the total category~$\Ecategory$ is a \emph{total cofibration}
when it factors as $X\to Z\to Y$ where $X\to Z$ is a cocartesian map above a cofibration $u\from A\to B$
of~$\Bcategory$, and $Z\to Y$ is a cofibration in the fiber~$\Ecategory_B$,
\item a morphism $f\from X\to Y$ of the total category~$\Ecategory$ is a \emph{total fibration}
when it factors as $X\to Z\to Y$ where $Z\to Y$ is a cartesian map 
above a fibration $u\from A\to B$ of~$\Bcategory$, and $X\to Z$ is a fibration in the fiber~$\Ecategory_A$.
\end{itemize}
%
%
%
\begin{proposition}[Uniqueness of the solution]
When the solution $(\class C_{\Ecategory},\class W_{\Ecategory},\class F_{\Ecategory})$ 
exists, it is uniquely determined by the fact that its fibrations and cofibrations
are the total cofibrations and total fibrations of the total category~$\Ecategory$, respectively.
\end{proposition}

Besides the formulation of Quillen bifibrations, our main contribution
is to devise two conditions called \eqref{hyp:weak-conservative} for
\emph{homotopical conservativity} and \eqref{hyp:hBC} for
\emph{homotopical Beck-Chevalley}, and to show (see
theorem~\ref{thm:main}) that they are sufficient and necessary for the
solution to exist.
%

\subsection{Related works}

The interplay between bifibred categories and model structures was
first explored by Roig in \cite{roig:model-bifibred}, providing
results in homological diffrentially graded algebra. Stanculescu then
spotted a mistake in Roig's theorem and subsequently corrected it in
\cite{stanculescu:bifib-model}. Finally,
\cite{harpaz-prasma:grothendieck-model} tackles the problem of
reflecting Lurie's Grothendieck construction for $\infty$-categories
at the level of model categories, hence giving a model for lax
colimits of diagrams of $\infty$-categories.

This work is directly in line with, and greatly inspired by, these
papers. In our view, both Roig-Stanculescu's and Harpaz-Prasma's
results suffer from flaws. The former introduces a very strong
asymmetry, making natural expectations unmet. For example, for any
Grothendieck bifibration $p\from \cat E \to \cat B$, the opposite
functor $\op p\from \op{\cat E} \to \op{\cat B}$ is also a
Grothendieck bifibration. So we shall expect that when it is possible
to apply Roig-Stanculescu's result to the functor $p$, provinding this
way a model structure on $\cat E$, it is also possible to apply it to
$\op p$, yielding on $\op{\cat E}$ the opposite model structure. This
is not the case: for almost every such $p$ for which the result
applies, it does not for the functor $\op p$. The latter result by
Harpaz and Prasma on the contrary forces the symmetry by imposing a
rather strong assumption: the adjoint pair $(\push u,\pull u)$
associated to a morphism $u$ of the base $\cat B$, already required to
be a Quillen adjunction in \cite{roig:model-bifibred} and
\cite{stanculescu:bifib-model}, needs in addition to be a Quillen
equivalence whenever $u$ is a weak equivalence. While it is a key
property for their applications, it put aside real world examples that
nevertheless satisfy the conclusion of the result. The goal of this
paper is to lay out a common framework fixing these flaws. This is
achieved in theorem \ref{thm:main} by giving necessary and sufficient
conditions for the resulting model structure on $\cat E$ to be the one
described in both cited results.

\paragraph{Plan of the paper.} Section \ref{sec:liminaries} recalls
the basic facts we will need latter about Grothendieck bifibrations
and model categories. It also introduces \define{intertwined weak
factorization systems}, a notion that pops here and there on forums
and the n\nobreakdash-Category Caf\'e, but does not appear in the
literature to the best of our knowledge. Its interest mostly resides
in that it singles out the 2-out-of-3 property of weak equivalences in
a model category from the other more {\em combinatorial}
properties. Finally we recall in that section a result of
\cite{stanculescu:bifib-model} in order to make this paper
self-contained.

Section \ref{sec:actual-thm} contains the main theorem~\ref{thm:main}
that we previously announced. Its proof is cut into two parts: first
we prove the necessity of conditions~\ref{hyp:weak-conservative} and
\ref{hyp:hBC}, and then we show that they are sufficient as well. The
proof of necessity is the easy part and comes somehow as a bonus,
while the proof of sufficiency is much harder and expose how
conditions~\eqref{hyp:weak-conservative} and \eqref{hyp:hBC} play
their role.

Section \ref{sec:examples} illustrates \ref{thm:main} with some
applications in usual homotopical algebra. First, it gives an original
view on Kan's theorem about Reedy model structures by stating it in a
bifibrationnal setting. Here should it be said that this was our
motivating example. We realized that neither Roig-Stanculescu's or
Harpaz-Prasma's theorem could be apply to the Reedy construction,
although the conclusion of these results was giving Kan's theorem
back. As in any of those {\em too good no to be true} situations, we
took that as an incentive to strip down the previous results in order
to only keep what makes them {\em tick}, which eventually has led to
the equivalence of theorem \ref{thm:main}. Section
\ref{subsec:versus-hp-rs} gives more details about Roig-Stanculescu's
and Harpaz-Prasma's theorem, and explains how their analysis started
the process of this work.

\paragraph{Convention.} All written diagrams commute if not said
otherwise. When objects are missing and replaced by a dot, they can be
parsed from other informations on the diagram. Gray parts help to
understand the diagram's context.

\paragraph{Acknowledgments.}
The authors are grateful to Clemens Berger for making them aware of
important references at the beginning of this work, and to Georges
Maltsiniotis for an early review of theorem~\ref{thm:main} and
instructive discussions around possible weakenings of the notion of
Quillen bifibration.

\section{Liminaries}
\label{sec:liminaries}

\subsection{Grothendieck bifibrations}
\label{subsec:lim-bifib}
In this section, we recall a number of basic definitions
and facts about Grothendieck bifibrations.

\medbreak

Given a functor $p:\Ecategory\to\Bcategory$, we shall use the following terminology.
The categories~$\Bcategory$ and~$\Ecategory$ are called
the \emph{basis category}~$\Bcategory$ and the \emph{total category}~$\Ecategory$
of the functor $p:\Ecategory\to\Bcategory$.
We say that an object $X$ of the total category~$\Ecategory$
is above an object $A$ of the basis category~$\Bcategory$
when $p(X)=A$ and, similarly, that a morphism $f:X\to Y$ 
is above a morphism $u:A\to B$ when $p(f)=u$.
The fiber of an object~$A$ in the basis category~$\Bcategory$
with respect to $p$ is defined as the subcategory of $\Ecategory$ 
whose objects are the objects~$X$ such that $p(X)=A$
and whose morphisms are the morphisms~$f$ such that $p(f)=\id A$.
In other words, the fiber of~$A$ is the category of objects above~$A$,
and of morphisms above the identity $\id A$.
The fiber is noted $p_A$ or $\Ecategory_A$ when no confusion is possible.

\medbreak

A morphism $f:X\to Y$ in a category~$\Ecategory$ is called \emph{cartesian}
with respect to the functor $p:\Ecategory\to\Bcategory$ when the commutative diagram
$$
\begin{tikzcd}[column sep=2em, row sep=1em]
\Ecategory(Z,X)
\arrow[rr,"{f\circ -}"] 
\arrow[dd,"{p}", swap] 
&& 
\Ecategory(Z,Y)
\arrow[dd,"p"] 
\\
\\
\Bcategory(C,A)
\arrow[rr,"{u\circ -}"] && 
\Bcategory(C,B)
\end{tikzcd}
$$
is a pullback diagram for every object $Z$ in the category~$\Ecategory$.
Here, we write~$u:A\to B$ and~$C$ for the images $u=p(f)$
and $C=p(Z)$ of the morphism~$f$ and of the object~$Z$, respectively.
Unfolding the definition, this means that for every pair of morphisms $v:C\to A$
and $g:Z\to Y$ above $u\circ v:C\to B$,
there exists a unique morphism $h:Z\to X$ above $v$ such that $h\circ f=g$.
The situation may be depicted as follows:
\begin{center}
$\begin{array}{c}
\begin{tikzcd}[column sep=1.2em, row sep=.8em]
Z
\arrow[rrrrrrd, bend left, "g"]
\arrow[rrd,dashed,"h"]
\arrow[dd,dotted,no head]
\\
&& 
X
\arrow[rrrr,"f"]
\arrow[dd,dotted,no head]
&&&& 
Y
\arrow[dd,dotted,no head]
\\
C
\arrow[rrrrrrd, bend left]
\arrow[rrd,"v"{swap}]
\\
&& A\arrow[rrrr,"u"{swap}]  &&&& B
\end{tikzcd}
\end{array}$
\end{center}
Dually, a morphism $f:X\to Y$ in a
category~$\Ecategory$ is called \emph{cocartesian} with respect to the
functor $p:\Ecategory\to\Bcategory$ when the commutative diagram
$$
\begin{tikzcd}[column sep=2em, row sep=1em]
\Ecategory(Y,Z)
\arrow[rr,"{-\circ f}"] 
\arrow[dd,"{p}", swap] 
&& 
\Ecategory(X,Z)
\arrow[dd,"p"] 
\\
\\
\Bcategory(B,C)
\arrow[rr,"{-\circ u}"] && 
\Bcategory(A,C)
\end{tikzcd}
$$
is a pullback diagram for every object $Z$ in the category~$\Ecategory$.
%
%
This means that for every pair of morphisms $v:A\to C$
and $g:X\to Z$ above $v\circ u:A\to C$,
there exists a unique morphism $h:Z\to X$ above~$v$ such that $h\circ f=g$.
%
Diagrammatically:
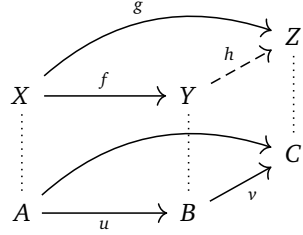
\begin{center}
$\begin{array}{c}
\begin{tikzcd}[column sep=1.2em, row sep=.8em]
&& &&&& 
Z
\arrow[dd,dotted,no head]
\\
X\arrow[rrrr,"f"] 
\arrow[rrrrrru, bend left, "g"]
\arrow[dd,dotted,no head]
&&&& 
Y
\arrow[rru,dashed,"h"]
\arrow[dd,dotted,no head]
\\
&& &&&& 
C
\\
A\arrow[rrrr,"u"{swap}] 
\arrow[rrrrrru, bend left]
&&&& 
B
\arrow[rru,"v"{swap}]
\end{tikzcd}
\end{array}$
\end{center}
A functor $p:\Ecategory\to\Bcategory$ is called a \emph{Grothendieck
  opfibration} when for every morphism $u:A\to B$ and for every
object~$Y$ above~$B$, there exists a cartesian morphism
$f:X\to Y$ above~$u$.
Symmetrically, a functor $p:\Ecategory\to\Bcategory$ is called a
\emph{Grothendieck opfibration} when for every morphism $u:A\to B$ and
for every object~$X$ above~$A$, there exists a cocartesian morphism
$f:X\to Y$ above~$u$.
Note that a functor $p:\Ecategory\to\Bcategory$ is a Grothendieck
opfibration precisely when the functor
$p^{op}:\Ecategory^{op}\to\Bcategory^{op}$ is a Grothendieck
fibration.
A \emph{Grothendieck bifibration} is a functor
$p:\Ecategory\to\Bcategory$ which is at the same time a Grothendieck
fibration and opfibration.
%

%
%

\begin{definition}
  A \define{cloven Grothendieck bifibration} is a functor
  $p \from \cat E \to \cat B$ together with
  \begin{itemize}
  \item for any $Y \in \cat E$ and $u \from A \to pY$, an object
    $\pull u Y \in \cat E$ and a cartesian morphism
    $\cart[p] u Y \from \pull u Y \to Y$ above $u$,
  \item for any $X \in \cat E$ and $u \from pX \to B$, an object
    $\push u X \in \cat E$ and a cocartesian morphism
    $\cocart[p] u X\from X \to \push u X$ above $u$.
  \end{itemize}
\end{definition}
When the context is clear enough, we might omit the index $p$. The
domain category $\cat E$ is often called the \define{total category}
of $p$, and its codomain $\cat B$ the \define{base category}. We shall
use this terminology when suited.
\begin{remark}
  If $\cat E$ and $\cat B$ are small relatively to a universe
  $\mathbb U$ in which we suppose the axiom of choice, then a cloven
  Grothendieck bifibration is exactly the same as the original notion
  of Grothendieck bifibration. Hence, in this article, we treat the
  two names as synonym.
\end{remark}


The data of such cartesian and cocartesian morphisms gives two
factorizations of an arrow $f\from X \to Y$ above some arrow
$u\from A \to B$,$\pushfact f$ in the fiber $\fiber {\cat E} B$ and
$\pullfact f$ in the fiber $\fiber {\cat E} A$: one goes through
$\cart u Y$ and the other through $\cocart u X$. See the diagram
below:
\begin{displaymath}
  \begin{tikzcd}
    X \ar[dr,"f"{description}] \ar[d,"\pullfact f"swap]
    \ar[r,""] &
    \push u X \ar[d,"\pushfact f"] \\
    \pull u Y \ar[r,""swap]& Y 
  \end{tikzcd}
\end{displaymath}
In turn, this allows $\push u$ and $\pull u$ to be extended as
\define{adjoint functors}:
\begin{displaymath}
  \push u \from \fiber{\cat E} A \adjointarrows
  \fiber{\cat E} B \cofrom \pull u
\end{displaymath}
where the action of $\push u$ on a morphism $k\from X \to X'$ of
$\fiber{\cat E} A$ is given by $\pushfact{(\cocart u {X'}\circ k)}$
and the action of $\pull u$ on a morphism $\ell \from Y' \to Y$ is
given by $\pullfact{(\ell \circ \cart u {Y'})}$:
\begin{displaymath}
  \begin{tikzcd}
    X \ar[d,"k"swap] \ar[r,""] &
    \push u X \ar[d,"\push u (k)"] \\
    X' \ar[r,""swap]& \push u {X'} 
  \end{tikzcd}
  \hskip 5em
  \begin{tikzcd}
    \pull u {Y'} \ar[d,"\pull u \ell"swap] \ar[r,""] &
    Y' \ar[d,"\ell"] \\
    \pull u Y \ar[r,""swap]& Y 
  \end{tikzcd}
\end{displaymath}
This gives a mapping $\cat B \to \Adj$ from the category $\cat B$ to
the $2$-category $\Adj$ of adjunctions: it maps an object $A$ to the
fiber $\fiber{\cat E}A$, and a morphism $u$ to the push-pull
adjunction $(\push u,\pull u)$. This mapping is even a pseudofunctor:
\begin{itemize}
\item For any $A \in\cat B$ and $X \in \fiber {\cat E} A$, we can
  factor $\id X \from X \to X$ through $\cocart {\id A}{X}$ and
  $\cart {\id A} X$:
  \begin{displaymath}
    \begin{tikzcd}
      X \drar[equal] \rar["\cocart {\id A}X"] & \push{(\id A)} X
      \dar["\pushfact{(\id X)}"]
      \\
      {} & X
    \end{tikzcd}
    \hskip 8em
    \begin{tikzcd}
      X \drar[equal] \dar["\pullfact{(\id
        X)}"swap] &
      \\
      \pull{(\id A)} X \rar["\cart {\id A}X"swap] & X
    \end{tikzcd}
  \end{displaymath}
  In particular by looking at the diagram on the left, both
  $\cocart {\id A} X \circ \pushfact{(\id X)}$ and the identity of
  $\push{(\id A)} X$ are solution to the problem of finding an arrow
  $f$ above $\id A$ such that $f\cocart{\id A} X = \cocart{\id A} X$:
  by the unicity condition of the cocartesian morphisms, it means that
  they are equal, or otherwise said that $\pushfact{(\id X)}$ is an
  isomorphism with inverse $\cocart{\id A}X$. Dually, looking at the
  diagram on the right, we deduce that $\pullfact{(\id X)}$ is an
  isomorphism with inverse $\cart{\id A} X$. All is natural in $X$, so
  we end up with
  \begin{displaymath}
    \push{(\id A)} \simeq \id {\fiber{\cat E}A} \simeq \pull{(\id A)}
  \end{displaymath}
\item For any $u\from A\to B$ and $v \from B \to C$ in $\cat B$, and
  for any $X \in \fiber{\cat E}A$, the cocartesian morphism
  $\cocart {vu} X\from X \to \push{(vu)} X$ is above $vu$ by
  definition hence should factorize as $h\cocart u X$ for some $h$
  above $v$, yielding $\pushfact h$ as a morphism in
  $\fiber{\cat E} C$ such that the following commutes:
  \begin{displaymath}
    \begin{tikzcd}
      X \rar["\cocart u X"] \ar[drr,bend right=15,"\cocart{vu}{X}"swap] &
      \push u X \rar["\cocart v {\push u X}"] \drar["h"swap] &
      \push v {\push u X} \dar["\pushfact h"]
      \\
      & & \push {(vu)} X
    \end{tikzcd}
  \end{displaymath}
  Writing simply $k$ for the composite
  $\cocart v {\push u X} \circ \cocart u X$, the following commutes:
  \begin{displaymath}
    \begin{tikzcd}
      X \drar["\cocart u X"swap] \ar[rr,"\cocart{vu}{X}"] & & \push
      {(vu)} X \ar[dd,"\pushfact k"]
      \\
      & \push u X \drar["\cocart v {\push u X}"swap] &
      \\
      & & \push v {\push u X}
    \end{tikzcd}
  \end{displaymath}
  Clearly $\pushfact h\pushfact k$ and $\id{\push v\push u X}$ both
  are solution to the problem of finding $f$ above $\id C$ such that
  $f\cocart{vu} X = \cocart{vu} X$: the uniqueness condition in the
  definition of cocartesian morphisms forces them to be
  equal. Conversely, we use the cocartesianness of $\cocart u X$ and
  $\cocart v {\push u X}$ in two steps: first
  $\pushfact k \pushfact h \cocart v {\push u X} = \cocart v {\push u
    X}$ because they both answer the problem of finding $f$ above $v$
  such that $f \cocart u X = \cocart v {\push u X} \circ \cocart u X$;
  from which we deduce
  $\pushfact k \pushfact h = \id{\push v \push u X}$ as they both
  answer the problem of finding a map $f$ above $\id C$ such that
  $f \cocart v {\push u X} = \cocart v{\push u X}$. In the end,
  $\pushfact h$ and $\pushfact k$ are isomorphisms, inverse to each
  other. All we did was natural in $X$, hence we have
  \begin{displaymath}
    \push {(vu)} \simeq \push v \push u
  \end{displaymath}
\item The dual argument shows that
  $\pull {(vu)} \simeq \pull u \pull v$.
\item To prove rigorously the pseudo functoriality of
  $\cat B \to \Adj$, we should show that the isomorphisms we have
  exhibited above are coherent. This is true, but irrelevant to this
  work, so we will skip it.
\end{itemize}

The pseudo functoriality relates through an isomorphism the chosen
(co)cartesian morphism above a composite $vu$ with the composite of
the chosen (co)cartesian morphisms above $u$ and $v$. The following
lemma gives some kind of extension of this result.
\begin{lemma}
  \label{lem:pseudo-push-iso}%
  Let $u\from A \to B$, $v\from B \to C$ and $w\from C \to D$ in
  $\cat B$.  Suppose $f \from X \to Y$ in $\cat E$ is above the
  composite $wvu$. Then for the unique maps
  $h\from\push v \push u X \to Y$ and $k\from \push {(vu)} X \to Y$
  above $w$ that fill the commutative triangles
  
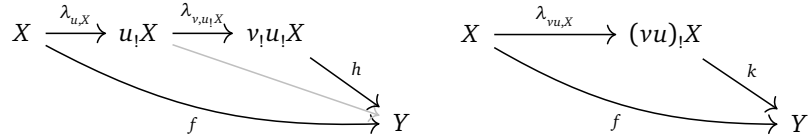
\begin{displaymath}
    \begin{tikzcd}
      X \rar["\cocart u X"] \ar[drrr,bend right=15,"f"swap] & \push u
      X \rar["\cocart v {\push u X}"] \ar[drr,lightgray] & \push v
      \push u X \drar["h"] &
      \\
      & & & Y
    \end{tikzcd}
    \quad%
    \begin{tikzcd}
      X \ar[rr,"\cocart {vu} X"] \ar[drrr,bend right=15,"f"swap] & &
      \push {(vu)} X \drar["k"] &
      \\
      & & & Y
    \end{tikzcd}
  \end{displaymath}
  there exists an isomorphism $\phi$ in the fiber $\fiber{\cat E}C$
  such that $h\phi = k$.
\end{lemma}
\begin{proof}
  We know there is a isomorphism
  $\phi \from \push{(vu)} X \to \push u \push v X$ above $\id C$ such
  that
  $\phi \cocart {vu} X = \cocart v {\push u X} \circ \cocart u X$. But
  then $h\phi \from \push{(vu)} X \to Y$ is above $w$ and fills the
  same triangle $k$ does in the statement: by unicity, $k = h\phi$.
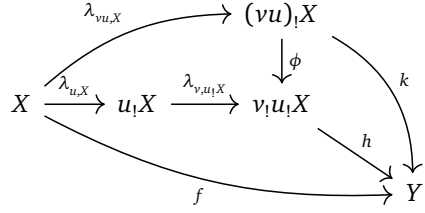
  \begin{displaymath}
    \begin{tikzcd}
      & & \push {(vu)} X \dar["\phi"] \ar[ddr,bend left,"k"] &
      \\
      X \ar[urr,"\cocart {vu} X",bend left=20] \rar["\cocart u X"]
      \ar[drrr,bend right=15,"f"swap] & \push u X \rar["\cocart v
      {\push u X}"] & \push v \push u X \drar["h"] &
      \\
      & & & Y
    \end{tikzcd}
  \end{displaymath}
\end{proof}

Of course, we have the dual statement, that accepts a dual proof.
\begin{lemma}
  \label{lem:pseudo-pull-iso}%
  Let $u\from A \to B$, $v\from B \to C$ and $w\from C \to D$ in
  $\cat B$.  Suppose $f \from X \to Y$ in $\cat E$ is above the
  composite $wvu$. Then for the unique maps
  $h\from X \to \pull u \pull v Y$ and $k\from X \to \pull {(vu)} Y$
  above $w$ that fill the commutative triangles
  
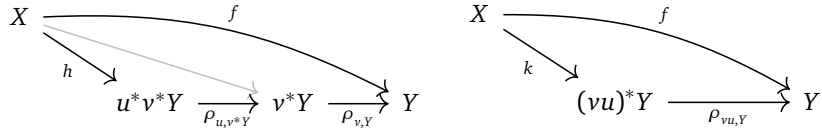
\begin{displaymath}
    \begin{tikzcd}
      X \ar[drrr,bend left=15,"f"] \ar[drr,lightgray] \drar["h"swap] & &
      &
      \\
      & \pull u \pull v Y \rar["\cart u {\pull v Y}"swap] & \pull v Y
      \rar["\cart v Y"swap] & Y
    \end{tikzcd}
    \quad%
    \begin{tikzcd}
      X \ar[drrr,bend left=15,"f"] \drar["k"swap] & &
      &
      \\
      & \pull {(vu)} Y \ar[rr,"\cart {vu}Y"swap] & & Y
    \end{tikzcd}    
  \end{displaymath}
  there exists an isomorphism $\phi$ in the fiber $\fiber{\cat E}C$
  such that $\phi k = h$.
\end{lemma}

Suppose now that we have a chain of composable maps in $\cat B$:
\begin{displaymath}
  \begin{tikzcd}
    A_0 \rar["u_1"] & A_1 \rar["u_2"] & \dots \rar["u_n"] & A_n
  \end{tikzcd}
\end{displaymath}
And let $f \from X \to Y$ be a map above the composite
$u_n\dots u_1u_0$. Choose $0\leq i,j\leq n$ such that $i+j\leq
n$. Then, using (co)cartesian choices above maps in $\cat B$, one can
construct two canonical maps associated to $f$: these are the unique maps
\begin{displaymath}
  \begin{aligned}
    h &\from \push {(u_i)} \cdots \push {(u_0)} X \to \pull
    {(u_{n-j+1})} \cdots \pull {(u_n)} Y
    \\
    &\text{and}
    \\
    k &\from \push{(u_i\cdots u_0)} X \to \pull{(u_n\cdots u_{n-j+1})} Y
  \end{aligned}
\end{displaymath}
above $u_{n-j}\cdots u_{i+1}\from A_i \to A_{n-j}$ (which is defined
as $\id {A_i}$ in case $i+j=n$) filling in the following commutative
diagrams:
\begin{displaymath}
  \begin{tikzcd}[column sep=small]
    X \rar["\cocartz"]
    & \dots \rar["\cocartz"] & \push {(u_i)} \cdots \push{(u_0)} X
    \ar[phantom,ddr,""{coordinate,name=Z}] & & &
    \ar[from=1-1,to=3-6,"f"swap,rounded corners,to
    path={(\tikztostart.south) |- (Z) [near end]\tikztonodes -|
      (\tikztotarget.north)}]
    \ar[from=1-3,to=3-4,"h"{fill=white},crossing over]
    \\
    & & & & &
    \\
    & & & \pull {(u_{n-j+1})} \cdots \pull {(u_n)} Y \rar["\cartz"]
    & \dots \rar["\cartz"]
    & Y
    \\
    &&&&&
    \\
    X \ar[rr,"\cocartz"]
    & & \push {(u_i\cdots u_0)} X
    \ar[phantom,ddr,""{coordinate,name=Z}] & & &
    \ar[from=5-1,to=7-6,"f"swap,rounded corners,to
    path={(\tikztostart.south) |- (Z) [near end]\tikztonodes -|
      (\tikztotarget.north)}]
    \ar[from=5-3,to=7-4,"k"{fill=white},crossing over]
    \\
    & & & & &
    \\
    & & & \pull {(u_n\cdots u_{n-j+1})} Y \ar[rr,"\cartz"]
    &
    & Y
  \end{tikzcd}
\end{displaymath}
By applying the previous lemmas multiples times, we get the following
useful corollary.
\begin{corollary}
  \label{cor:pseudo-iso-in-fiber}%
  There is fiber isomorphisms $\phi$ and $\psi$ such that the
  following commutes:
  \begin{displaymath}
    \begin{tikzcd}[column sep=small]
      X \rar["\cocartz"]
      & \dots \rar["\cocartz"] & \push {(u_i)} \cdots \push{(u_0)} X
      \ar[phantom,ddr,""{coordinate,name=Z}] & & &
      \ar[from=1-1,to=3-6,"f",rounded corners,to
      path={(\tikztostart.south) |- (Z) -| (\tikztotarget.north) [near
        start]\tikztonodes}]
      \ar[from=1-1,to=4-3,"\cocartz"{swap},crossing over]
      \ar[from=1-3,to=3-4,"h"{fill=white},crossing over]
      \\
      & & & & &
      \\
      & & & \pull {(u_{n-j+1})} \cdots \pull {(u_n)} Y \rar["\cartz"] &
      \dots \rar["\cartz"]
      & Y
      \\
      & & \push {(u_i\cdots u_0)} X
      \ar[phantom,ddr,""{coordinate,name=Z}] \ar[uuu,"\phi",crossing
      over] & & &
      \ar[from=4-3,to=6-4,"k"{fill=white},crossing over]
      \\
      & & & & &
      \\
      & & & \pull {(u_n\cdots u_{n-j+1})} Y \ar[uuu,"\psi",crossing
      over] 
      &
      & 
      \ar[from=6-4,to=3-6,"\cartz"{swap}]
    \end{tikzcd}
  \end{displaymath}
\end{corollary}
We will extensively use this corollary when $i+j = n$. Indeed, in that
case $h,k,\phi,\psi$ all are in the same fiber $\fiber{\cat E}{A_i}$
and then $h$ and $k$ are isomorphic as arrows in that fiber. Every
property on $h$ that is invariant by isomorphism of arrows will still
hold on $k$, and conversely.

\subsection{Weak factorization systems}%
\label{subsec:wfs-intertwined}%

In any category $\cat M$, we denote $j \weakorth q$, and we say that
$j$ has the left lifting property relatively to $q$ (or that $q$ has
the right lifting property relatively to $j$), when for any
commutative square of the form
\begin{displaymath}
  \begin{tikzcd}
    A \ar[r] \ar[d,"j" swap] & C \ar[d,"q"] \\
    B \ar[r] & D
  \end{tikzcd}
\end{displaymath}
there exists a morphism $h \from B \to C$, making
the two triangles commute in the following diagram:
\begin{displaymath}
  \begin{tikzcd}
    A \ar[r] \ar[d,"j" swap] & C \ar[d,"q"] \\
    B \ar[r] \ar[ur,"h" description] & D
  \end{tikzcd}
\end{displaymath}
Such a morphism~$h$ is called a \emph{lift} of the original commutative square.

\medbreak

A \define{weak factorization system} on a category $\cat M$ is the
data of a couple $(\class L,\class R)$ of classes of arrows in
$\cat M$ such that
\begin{displaymath}
  \class L = \{ j : \forall q \in R, j\weakorth q\}
  \quad \text{and} \quad
  \class R = \{ q : \forall j \in L, j\weakorth q\}
\end{displaymath}
and such that every morphism $f$ of $\cat M$ may be factored as $f=qj$
with $j\in \class L,q\in \class R$. The elements of $\class L$ are
called the \define{left maps} and the elements of $\class R$ the
\define{right maps} of the factorization system.

Let now $\cat M$ and $\cat N$ be categories with both a factorization
system. Then an adjunction
$L \from \cat M \rightleftarrows \cat N \cofrom R$ is said to be
\define{wfs-preserving} if the left adjoint $L$ preserves the left
maps, or equivalently if the right adjoint $R$ respects the right
maps.

\medskip

As a key ingredient in the proof of our main result, the following
lemma deserves to be stated fully and independently. It explains how
to construct a weak factorization system on the total category of a
Grothendieck bifibration, given that the basis and fibers all have one
in a way that the adjunctions arising from the bifibration are
wfs-preserving.

\begin{lemma}[Stanculescu]%
  \label{lem:stan-lemma}%
  Let $\pi \from \cat F \to \cat C$ be a Grothendieck bifibration with
  weak factorization systems $(\class L_C,\class R_C)$ on each fiber
  $\fiber{\cat F}C$ and $(\class L, \class R)$ on $\cat C$. If the
  adjoint pair $(\push u,\pull u)$ is a wfs-adjunction for every
  morphism $u$ of $\cat C$, then there is a weak factorization system
  $(\class L_{\cat F},\class R_{\cat F})$ on $\cat F$ defined by
  \begin{displaymath}
    \begin{aligned}
      \class L_{\cat F} &= \{ f \from X \to Y \in \cat F : \pi(f) \in
      \class L , \pushfact f \in \class L_{\pi Y} \},
      \\
      \class R_{\cat F} &= \{ f \from X \to Y \in \cat F : \pi(f) \in
      \class R , \pullfact f \in \class R_{\pi X} \}
    \end{aligned}
  \end{displaymath}
\end{lemma}
For the proof in \cite[2.2]{stanculescu:bifib-model} is based on a
different (yet equivalent) definition of weak factorization systems,
here is a proof in our language for readers's convinience.
\begin{proof}
  Let us begin with the easy part, which is the factorization
  property. For a map $f\from X \to Y$ of $\cat F$, one gets a
  factorization $\pi(f) = r \ell$ in $\cat C$ with
  $\ell \from \pi X \to C \in \class L$ and
  $r\from C \to \pi Y\in \class R$. It induces a fiber morphism
  $\push \ell X \to \pull r Y$ in $\fiber{\cat F}C$ that we can in
  turn factor as $r_C\ell_C$ with $\ell_C \in \class L_C$ and
  $r_C \in \class R_C$.
  \begin{displaymath}
    \begin{tikzcd}
      X \rar \drar[densely dashed,"\tilde \ell",swap]
      & \push \ell X \dar["\ell_C"] & \\
      & \cdot \dar["r_C",swap] \drar[densely dashed,"\tilde r"] & \\
      & \pull r Y \rar & Y
    \end{tikzcd}
  \end{displaymath}
  Then the wanted factorization of $f$ is $\tilde r \tilde \ell$ where
  $\tilde r$ is the morphism of $\cat F$ such that $\pi(\tilde r) = r$
  and $\pullfact {\tilde r} = r_C$, and $\tilde \ell$ the one such
  that $\pi (\tilde \ell) = \ell$ and
  $\pushfact{\tilde \ell} = \ell_C$. This is summed up in the previous
  diagram.

  Lifting properties follow the same kind of pattern: take the image
  by $\pi$ and do the job in $\cat C$, then push and pull in $\cat F$
  so that you end up in a fiber when everything goes smoothly. Take a
  map ${j \from X \to Y} \in \class L_{\cat F}$ and let us show that
  it lift against elements of $\class R_{\cat F}$. Consider in
  $\cat F$ a commutative square with the map $q$ on the right in
  $\class R_{\cat F}$:
  \begin{displaymath}
    \begin{tikzcd}
      X \dar["j",swap] \rar["f"] & V \dar["q"] \\
      Y \rar["g",swap] & W
    \end{tikzcd}
  \end{displaymath}
  By definition, $\pi (j) \in \class L$ has the left lifting property
  against $\pi(q)$, hence a lift $h$:
  \begin{displaymath}
    \begin{tikzcd}
      \pi X \dar["\pi(j)",swap] \rar["\pi(f)"] & \pi V \dar["\pi(q)"] \\
      \pi Y \rar["\pi(g)",swap] \urar["h"] & \pi W
    \end{tikzcd}
  \end{displaymath}
  Now filling the original square with $\tilde h \from Y \to V$ above
  $h$ is equivalent to fill the following induced solid square in
  $\fiber{\cat F} {\pi Y}$:
  \begin{displaymath}
    \begin{tikzcd}
      \color{lightgray} X \drar["j",swap,lightgray]
      \ar[rrr,bend left,"f",lightgray] \rar[lightgray] &
      \cdot \dar["\pushfact j",swap] \rar
      & \cdot \dar["\pull h (\pullfact q)"] \rar[lightgray] &
      \color{lightgray} V \drar[lightgray,"q"] \dar[lightgray,"\pullfact q"] & \\
      & Y \rar \ar[rrr,bend right,"g",swap,lightgray] & \cdot \rar[lightgray]
      & \color{lightgray} \cdot \rar[lightgray] & \color{lightgray} W
    \end{tikzcd}
  \end{displaymath}
  But $\pushfact j \in \class L_{\pi Y}$, and $\pull h$ is the right
  adjoint of a wfs-preserving adjunction, hence maps the right map
  $\pullfact q$ of $\fiber{\cat F}{\pi V}$ to a right map in
  $\fiber{\cat F}{\pi Y}$: so there is such a filler.

  Conversely, if $j \from X \to Y$ in $\cat F$ has the left lifting
  property relatively to all maps of $\class R_{\cat F}$, then one has
  to show that it is in $\class L_{\cat F}$. Consider in
  $\fiber{\cat F}{\pi Y}$ a commutative square as
  \begin{displaymath}
    \begin{tikzcd}
      \color{lightgray} X \rar[lightgray,"\cocart {\pi(j)} X"] \drar["j",swap,lightgray]
      & \cdot \rar["f"] \dar["\pushfact j"] & Y' \dar["q"] \\
      & Y \rar["g",swap] & Y''
    \end{tikzcd}
    \qquad q \in \class R_{\pi Y}
  \end{displaymath}
  Then, because $q$ also is in $\class R_{\cat F}$, there is an
  $h \from Y \to Y'$ such that $g = qh$ and $hj = f\cocart{\pi(j)}
  X$. But then, $h\pushfact j$ and $f$ both are solution to the
  factorization problem of $j$ through the cocartesian arrow
  $\cocart{\pi(j)} X$, hence should be equal. Meaning $h$ is a filler
  of the original square in the fiber $\fiber{\cat F}{\pi Y}$. We
  conclude that $\pushfact j$ is a left map in its fiber. Now consider
  a commutative square in $\cat C$:
  \begin{displaymath}
    \begin{tikzcd}
      \pi X \dar["\pi(j)",swap] \rar["f"] & C \dar["q"] \\
      \pi Y \rar["g",swap] & D
    \end{tikzcd}
    \qquad q \in \class R
  \end{displaymath}
  It induced a commutative square in $\cat F$:
  \begin{displaymath}
    \begin{tikzcd}
      X \dar["j",swap] \rar & \pull q \push g Y \dar["\kappa"] \\
      Y \rar & \push g Y
    \end{tikzcd}
  \end{displaymath}
  Now the arrow on the right is cartesian above a right map, hence is
  in $\class R_{\cat F}$ by definition. So $j$ lift against it, giving
  us a filler $h \from Y \to \pull q \push g Y$ whose image
  $\pi(h) \from Y \to C$ fills the square in $\cat C$. We conclude
  that $\pi(j)$ is a left map of $\cat C$. In the end,
  $j \in \class L_{\cat F}$ as we wanted to show.

  Similarly, we can show that $\class R_{\cat F}$ is exactly the class
  of maps that have the right lifting property against all maps of
  $\class L_{\cat F}$.
\end{proof}

\subsection{Intertwined weak factorization system and model
  categories}%
\label{subsec:lim-model}%

Quillen introduced model categories in
\cite{quillen:homotopical-algebra} as categories with sufficient
structural analogies with the category of topological spaces so that a
sensible notion of {\em homotopy between maps} can be provided. Not
necessarily obvious at first sight are the redundancies of Quillen's
definition. Even though intentionally important in the conceptual
understanding of a model category, the extra checkings required can
make a simple proof into a painful process. To ease things a little
bit, this part is dedicated to extract the minimal definition of a
model category at the cost of trading topological intuition for
combinatorial comfort.

\medskip

Recall the definition of a model structure.
\begin{definition}%
  \label{def:model-category}%
  A model structure on a category $\cat M$ is the data of three
  classes of maps $\class C$, $\class W$, $\class F$ such that:
  \begin{enumerate}[label=(\roman*)]
  \item $\class W$ has the 2-out-of-3 property, i.e.\ if two elements
    among $\{f,g,gf\}$ are in $W$ for composable morphisms $f$ and
    $g$, then so is the third,
  \item $(\class C,\class W \cap \class F)$ and
    $(\class C \cap \class W,\class F)$ both are weak fatorization
    systems.
  \end{enumerate}
\end{definition}
The morphism in $\class W$ are called the \define{weak equivalences},
those in $\class C$ the \define{cofibrations} and those in $\class F$
the \define{fibrations.} Given the role played by the two classes
$\class C \cap W$ and $\class F \cap \class W$, we also give names to
their elements: a fibration (respectively a cofibration) which is also
a weak equivalence is called an \define{acyclic fibration}
(respectively an \define{acyclic cofibration}).

\begin{remark}
  It is crucial for the rest of the document to remark that there is
  some redundancy in the previous definition: in a model structure,
  any two of the three classes $\class C,\class W,\class F$ determine
  the last one. Indeed, knowing of $\class C$ and $\class W$ gives us
  $\class F$ as the class of morphisms having the right lifting
  property relatively to every element of $\class C \cap \class
  W$. Dually, $\class C$ is given as the class of morphism having the
  left lifting property relatively to every element of
  $\class W \cap \class F$.

  Finally, and it is the relevant case for the purpose of this
  article, the weak equivalences are exactly those morphisms that we
  can write $qj$ where $j$ is an acyclic cofibration and $q$ is an
  acyclic fibration. The first inclusion
  $\{qj: q \in \class F \cap \class W, j \in \class C \cap \class W \}
  \subseteq \class W$ is a direct consequence of the 2-out-of-3
  property. The converse inclusion is given by applying one of the two
  weak factorization systems and then using the 2-out-of-3 property:
  if $w\in \class W$, it is writable as $w=qj$ with
  $q \in \class F, j\in \class C \cap \class W$; but then $w$ and $j$
  being a weak equivalence, $q$ also is. Hence the conclusion.
\end{remark}

Recall also the notion of morphisms between model structures: a
\define{Quillen adjunction} between two model structures $\cat M$ and
$\cat N$ is an adjunction
${L \from \cat M \rightleftarrows \cat N \cofrom R}$ which is
wfs-preserving for both the weak factorization system (acyclic
cofibrations, fibrations) and the (cofibrations, acyclic fibrations)
one.

Finally, to conclude those remainders about model structures, let us
introduce some new vocabulary.
\begin{definition}[Homotopically conservative functor]
  A functor $F \from \cat M \to \cat N$ between model structures is
  said to be \define{homotopically conservative} if it preserves and
  reflects weak equivalences.
\end{definition}

\begin{remark}
  To get one's head around this terminlogy, let us make two
  observations:
  \begin{enumerate}[label=(\arabic*)]
  \item If $\cat M$ and $\cat N$ are endowed with the trivial model
    structure, in which weak equivalences are isomorphisms and
    cofibrations and fibrations are all morphisms, then the notion
    boils down to the usual conservative functors.
  \item Every functor $F \from \cat M \to \cat N$ preserving weak
    equivalences induces a functor
    $\Ho F \from \Ho {\cat M} \to \Ho {\cat N}$. Given that weak
    equivalences are saturated in a model category, homotopically
    conservative functors are exactly those $F$ such that $\Ho F$ is
    conservative as a usual functor.
  \end{enumerate}
\end{remark}

\medskip

Let us pursue with the following definition, apparently absent from
literature.
\begin{definition}
  A weak factorization system $(\class L_1,\class R_1)$ on a category
  $\cat C$ is \define{intertwined} with another
  $(\class L_2,\class R_2)$ on the same category when:
  \begin{displaymath}
    \class L_1 \subseteq \class L_2
    \qquad\text{ and }\qquad
    \class R_2 \subseteq \class R_1.
  \end{displaymath}
\end{definition}
The careful reader will notice that the properties
$\class L_1 \subseteq \class L_2$ and
$\class R_2 \subseteq \class R_1$ are actually equivalent to each
other, but the definition is more naturally stated in this way.
A similar notion is formulated by Shulman for orthogonal factorization systems,
in a blog post on the $n$-Category Caf\'e \cite{shulman:ncat-ternary}
with a brief mention at the end of a version for weak factorization systems. 
This is the only appearance of such objects known to us.

The similarity with the weak factorization systems of a model category
is immediately noticeable and in fact it goes further than a mere
resemblance, as indicated in the following two results.
\begin{proposition}
  Let $(\class L_1,\class R_1)$ together with
  $(\class L_2,\class R_2)$ form intertwined weak factorization
  systems on a category $\cat C$. Denoting
  $\class W = \class R_2 \circ \class L_1$, the following class
  identities hold:
  \begin{displaymath}
    \class L_1 = \class W \cap \class L_2, \qquad
    \class R_2 = \class W \cap \class R_1.
  \end{displaymath}
\end{proposition}
\begin{proof}
  Let us prove the first identity only, as the second one is strictly
  dual. Suppose $f\from A \to B \in \class L_1$, then
  $f \in \class L_2$ by the very definition of intertwined weak
  factorization systems, and $f = \id B f \in \class W$, hence the
  first inclusion: $\class L_1 \subseteq \class W \cap \class
  L_2$.

  Conversely, take $f \in \class W \cap \class L_2$. Then in
  particular there exists $j\in \class L_1$ and $q \in \class R_2$
  such that $f = qj$. Put otherwise, the following square commutes:
  \begin{displaymath}
    \begin{tikzcd}
      A \ar[d,"f",swap] \ar[r,"j"] & C \ar[d,"q"] \\
      B \ar[r,equal] & B.
    \end{tikzcd}
  \end{displaymath}
  But $f$ is in $\class L_2$ and $q$ is in
  $\class R_2 \subseteq \class R_1$, hence a lift $s \from B \to C$
  such that $qs = \id B$ and $sf = j$. Now for any $p \in \class R_1$
  and any commutative square
  \begin{displaymath}
    \begin{tikzcd}
      A \ar[dd,"f"swap,bend right] \ar[d,"j"] \ar[r,"x"]
      & D \ar[dd,"p"] \\
      C \ar[d,"q"]& \\
      B \ar[r,"y"] & E
    \end{tikzcd}
  \end{displaymath}
  there is a lift $h \from C \to D$ taking advantage of $j$ having the
  left lifting property against $p$. Then $hs \from B \to D$ provides
  a lift showing that $f$ has the left lifting property against $p$:
  indeed $phs = yqs = y$ and $hsf = hsqj = hj = x$. Having the left
  lifting property against any morphism in $\class R_1$, the morphism
  $f$ ought to be in $\class L_1$, hence providing the reverse
  inclusion: $\class W \cap \class L_2 \subseteq \class L_1$.
\end{proof}

\begin{corollary}
  \label{cor:intertwined-model}%
  Let $(\class L_1,\class R_1)$ and $(\class L_2,\class R_2)$ form
  intertwined weak factorization systems on a category $\cat M$, and
  denote again $\class W = \class R_2 \circ \class L_1$. The category
  $\cat M$ has a model structure with weak equivalences $\class W$,
  fibrations $\class R_1$ and cofibrations $\class L_2$ if and only if
  $\class W$ has the 2-out-of-3 property.
\end{corollary}
Of course in that case, we also get the class of acyclic cofibrations
as $\class L_1$ and the class of acyclic fibrations as $\class R_2$.

\medskip

So there it is: we shreded apart the notion of a model structure to
the point that what remains is the pretty tame notion of intertwined
factorization systems $(\class L_1,\class R_1)$ and
$(\class L_2,\class R_2)$ such that $\class R_2\circ\class L_1$ has
the 2-out-of-3 property. But it has the neat advantage to be easily
checkable, especially in the context of formal constructions, as it is
the case in this paper. It also emphasizes the fact that Quillen
adjunctions are really the right notion of morphisms for intertwined
weak factorization systems and have {\em a priori} nothing to do with
weak equivalences. We shall really put that on a stand because
everything that follows in the main theorem can be restated with mere
intertwined weak factorization systems in place of model structures
and it still holds: in fact it represents the easy part of the theorem
and all the hard core of the result resides in the 2-out-of-3
property, as usually encountered with model structures.

\section{Quillen bifibrations}
\label{sec:quillen-bifib}

Recall from the introduction that a \define{Quillen bifibration} is a
Grothendieck bifibration $p \from \cat E \to \cat B$
between categories with model structures such that:
\begin{enumerate}[label=(\roman*)]
\item the functor~$p$ is both a left and right Quillen functor,
\item the model structure on $\cat E$ restricts to a model structure
on the fiber $\fiber{\cat E}A$, for every object $A$ of the category~$\cat B$.
\end{enumerate}
In this section, we show that in a Quillen bifibration the model
structure on the basis~$\cat B$ and on every fiber $\cat E_A$
determines the original model structure on the total category $\cat E$.
In the remainder of this section, we fix a Quillen bifibration
$p\from \cat E \to \cat B$.

\begin{lemma}
  \label{lem:quillen-bifib-push-pull-are-quillen}%
  For every morphism $u \from A \to B$ in $\cat B$, the adjunction
  $\push u \from \fiber{\cat E}A \rightleftarrows \fiber{\cat E}B
  \cofrom \pull u$ is a Quillen adjunction.
\end{lemma}
\begin{proof}
  Let $f\from X\to Y$ be a cofibration in the fiber $\fiber{\cat
    E}A$. We want to show that the morphism $\push u (f)$ of
  $\fiber{\cat E}B$ is a cofibration. Take an arbitrary acyclic
  fibration $q \from W \to Z$ in $\fiber{\cat E} B$ and a commutative
  square in that fiber:
  \begin{displaymath}
    \begin{tikzcd}
      \push u X \rar["g"] \dar["\push u (f)"swap] & W \dar["q"]
      \\
      \push u Y \rar["g'"swap] & Z
    \end{tikzcd}
  \end{displaymath}
  We need to find a lift $h \from \push u Y \to W$ making the diagram
  commutes, i.e.\ such that $qh = g'$ and $h\push u(f) = g$. Let us
  begin by precomposing with the square defining $\push u(f)$:
  \begin{displaymath}
    \begin{tikzcd}
      X \rar["\cocartz"] \dar["f"swap] & \push u X \rar["g"]
      \dar["\push u (f)"swap] & W \dar["q"]
      \\
      Y \rar["\cocartz"swap] & \push u Y \rar["g'"swap] & Z
    \end{tikzcd}
  \end{displaymath}
  As a cofibration, $f$ has the left lifting property against $q$,
  providing a map ${k\from Y \to W}$ that makes the following commute:
  \begin{displaymath}
    \begin{tikzcd}[row sep=large]
      X \rar["\cocartz"] \dar["f"swap] & \push u X \rar["g"]
      \dar[lightgray,"\push u (f)"{swap,near start}] & W \dar["q"]
      \\
      Y \rar["\cocartz"swap] \ar[urr,"k"{near start},crossing over] &
      \push u Y \rar["g'"swap] & Z
    \end{tikzcd}
  \end{displaymath}
  Now we use the cocartesian property of
  $\cocart u Y \from Y \to \push u Y$ on $k$, to find a map
  $h\from \push u Y \to W$ above the identity $\id B$ such that
  $h\cocart u Y = k$. All it remains to show is that $qh = g'$ and
  $h\push u (f) = g$. Notice that both $qh$ and $g'$ answer to the
  problem of finding a map $x \from \push u Y \to Z$ above $\id B$
  such that $x\cocart u Y = qk$: hence, by the unicity condition in
  the cocartesian property of $\cocart u Y$, they must be
  equal. Similarly, $h\circ\push u (f)$ and $g$ solve the problem of
  finding $x \from \push u X \to W$ above $\id B$ such that
  $x\cocart u X = kf$: the cocartesian property of $\cocart u X$
  allows us to conclude that they are equal. In the end, $\push u(f)$
  has the left lifting property against every acyclic fibration of
  $\fiber{\cat E}B$, so it is a cofibration. We prove dually that the
  image $\pull u f$ of a fibration $f$ in $\fiber{\cat E} B$ is a
  fibration of the fiber $\fiber{\cat E} A$.
\end{proof}

\begin{lemma}
  \label{lem:quill-bifib-cocart-above-cof}%
  A cocartesian morphism in $\cat E$ above a (acyclic) cofibration of
  $\cat B$ is a (acyclic) cofibration.
\end{lemma}
\begin{proof}
  Let $f\from X \to Y$ be cocartesian above a cofibration
  $u \from A \to B$ in $\cat B$.
  Given a commutative square of $\cat E$
  \begin{equation}
    \label{eq:first-square}%
    \begin{tikzcd}
      X \dar["f"{swap}] \rar["g"] & W \dar["q"]
      \\
      Y \rar["g'"{swap}] & Z
    \end{tikzcd}
  \end{equation}
  with $q$ an acyclic fibration, we can take its image in $\cat B$:
  \begin{displaymath}
    \begin{tikzcd}
      A \dar["u"{swap}] \rar & pW \dar["p(q)"]
      \\
      B \rar & pZ
    \end{tikzcd}
  \end{displaymath}
  Since $u$ is a cofibration and $p(q)$ an acyclic fibration, there
  exists a morphism~$h \from B \to pW$ making the expected diagram commute:
    \begin{displaymath}
    \begin{tikzcd}
      A \dar["u"{swap}] \rar & pW \dar["p(q)"]
      \\
      B \rar \arrow[ru,"h" description] & pZ
    \end{tikzcd}
  \end{displaymath}
  Because $f$ is cocartesian, we know that there exists a (unique) map
  $\tilde h \from Y \to W$ above $h$ making the diagram below commute:
  \begin{displaymath}
    \begin{tikzcd}
      X \dar["f"{swap}] \rar["g"] & W
      \\
      Y \urar["\tilde h"{swap}] &
    \end{tikzcd}
  \end{displaymath}
  For the morphism~$\tilde h$ to be a lift in the first commutative
  square~(\ref{eq:first-square}), there remains to show
  that~$q\tilde h = g'$. Because $\tilde h$ is above $h$ and
  $p(q)h = p(g')$, we have that the composite~$q\tilde h$ is above
  $g'$. Moreover $q\tilde h f = q g = g' f$. Using the uniqueness
  property in the universal definition of cocartesian maps, we deduce
  $q\tilde h = g'$. We have just shown that the cocartesian
  morphism~$f$ is weakly orthogonal to every acyclic fibration, and we
  thus conclude that $f$ is a cofibration. The case of cocartesian
  morphisms above acyclic cofibrations is treated in a similar way.
\end{proof}
The same argument establishes the dual statement:
\begin{lemma}
  \label{lem:quill-bifib-cart-above-fib}%
  A cartesian morphism in $\cat E$ above a (acyclic) fibration of
  $\cat B$ is a (acyclic) fibration.
\end{lemma}

\begin{proposition}
  \label{prop:quill-bifib-charac-total-cof}%
  A map $f\from X \to Y$ in $\cat E$ is a (acyclic) cofibration if and
  only if $p(f)$ is a (acyclic) cofibration in $\cat E$ and
  $\pushfact f$ is a (acyclic) cofibration in the fiber
  $\fiber{\cat E}{pY}$.
\end{proposition}
\begin{proof}
  A direction of the equivalence is easy: if $p(f) = u \from A\to B$
  is a cofibration, then so is the cocartesian morphism $\cocart u X$
  above it by lemma \ref{lem:quill-bifib-cocart-above-cof}; if
  moreover $\pushfact f$ is a cofibration in the fiber
  $\fiber{\cat E}B$, then $f = \pushfact f \cocart u X$ is a composite
  of cofibration, hence it is a cofibration itself.

  Conversely, suppose that $f\from X\to Y$ is a cofibration in $\cat
  E$. Then surely $p(f) = u\from A\to B$ also is a cofibration in
  $\cat B$, since $p$ is a left Quillen functor. Now we want to show
  that $\pushfact f \from \push u X \to Y$ is a cofibration in the
  fiber $\fiber{\cat E}B$. Consider a commutative square in that fiber
  \begin{displaymath}
    \begin{tikzcd}
      \push u X \rar["g"] \dar["\pushfact f"{swap}] & W \dar["q"]
      \\
      Y \rar["g'"{swap}] & Z
    \end{tikzcd}
  \end{displaymath}
  where $q$ is an acyclic fibration of the fiber $\fiber{\cat E}B$,
  and $g,g'$ are arbitrary morphisms in that fiber.
  Since $f$ itself is a cofibration in~$\Ecategory$,
  we know that there exists a lift $h \from Y \to W$
  for the outer square (with four sides $f$, $q$, $g\cocart u X$ and $g'$)
  of the following diagram:
  \begin{displaymath}
    \begin{tikzcd}\cocart u X
      X \drar["f"{swap}] \rar["\cocart u X"] & \push u X \rar["g"]
      \dar[lightgray,"\pushfact f"] & W \dar["q"]
      \\
      & Y \rar["g'"{swap}] \urar["h"] & Z
    \end{tikzcd}
  \end{displaymath}
  Now, there remains to show that $h\pushfact f = g$. We already know that
  $g\cocart u X = h \pushfact f \cocart u X$, and taking advantage of
  the fact that the morphism $\cocart u X$ is cocartesian, 
  we only need to show that $p(g) = p(h \pushfact f)$. 
  Since $g$ and $\pushfact f$ are fiber
  morphisms, it means we need to show that $h$ also. This follows from
  the fact that $qh = g'$ and that $q$ and $g'$ are fiber morphisms.
\end{proof}

In the same way, we get the dual statement:
\begin{proposition}
  \label{prop:quill-bifib-charac-total-fib}%
  A map $f\from X \to Y$ in $\cat E$ is a (acyclic) cofibration if and
  only if $p(f)$ is a (acyclic) cofibration in $\cat E$ and
  $\pushfact f$ is a (acyclic) cofibration in the fiber
  $\fiber{\cat E}{pY}$.
\end{proposition}

In particular, this means that the model structure on the total category
$\cat E$ is entirely determined by the model structures on the basis
$\cat B$ and on each fiber~$\cat E_B$ of the bifibration. As these
characterizations turn out to be important for what follows, we shall
name them.
\begin{definition}
  \label{def:total-model}%
  Let $p\from \cat E \to \cat B$ be a Grothendieck bifibration such
  that its basis $\cat B$ and each fiber
  $\fiber{\cat E}A\ (A\in\cat B)$ have a model structure.
  \begin{itemize}
  \item a \define{total cofibration} is a morphism $f\from X \to Y$ of
    $\cat E$ above a cofibration $u\from A \to B$ of $\cat B$ such
    that $\pushfact f$ is a cofibration in the fiber
    $\fiber{\cat E} B$,
  \item a \define{total fibration} is a morphism $f\from X \to Y$ of
    $\cat E$ above a fibration $u\from A \to B$ of $\cat B$ such that
    $\pullfact f$ is a fibration in the fiber $\fiber{\cat E} A$,
  \item a \define{total acyclic cofibration} is a morphism
    $f\from X \to Y$ of $\cat E$ above an acyclic cofibration
    $u\from A \to B$ of $\cat B$ such that $\pushfact f$ is an acyclic
    cofibration in the fiber $\fiber{\cat E} B$,
  \item a \define{total acyclic fibration} is a morphism
    $f\from X \to Y$ of $\cat E$ above an acyclic fibration
    $u\from A \to B$ of $\cat B$ such that $\pullfact f$ is an acyclic
    fibration in the fiber $\fiber{\cat E} A$.
  \end{itemize}
\end{definition}
Using this terminology, the two propositions
\ref{prop:quill-bifib-charac-total-cof} and
\ref{prop:quill-bifib-charac-total-fib} just established
come together as: the cofibrations, fibrations, acyclic cofibrations 
and acyclic fibrations of a Quillen bifibration are necessarily the total ones.
Note also that the definitions of total cofibration and total
fibration given in definition~\ref{def:total-model} coincides with the
definition given in the introduction.

\medskip

We end this section by giving simple examples of Quillen
bifibrations. They should serve as both a motivation and a guide for
the reader to navigate into the following definitions and proofs: it
surely has worked that way for us authors.
\begin{example}%
  \label{ex:main-thm-motiv}%
  ~%
  \begin{enumerate}[label=(\arabic*)]%
  \item One of the simplest instances of a Grothendieck bifibration
    other than the identity functor, is a projection from a product:
    \begin{displaymath}
      p\from \cat M \times \cat B \to \cat B
    \end{displaymath}
    Cartesian and cocartesian morphisms coincide and are those of the
    form $(\id M,u)$ for $M \in \cat M$ and $u$ a morphism of
    $\cat B$. In particular, one have $\pullfact {(f,u)} = (f,\id A)$
    and $\pushfact {(f,u)} = (f,\id B)$ for any $u \from A \to B$ in
    $\cat B$ and any $f$ in $\cat M$.

    If $\cat B$ and $\cat M$ are model categories, each fiber
    $\fiber p A \simeq \cat M$ inherits a model structure from
    $\cat M$ and the total fibrations and cofibrations coincide
    precisely with the one of the usual model structure on the product
    $\cat M \times \cat B$.
  \item For a category $\cat B$, one can consider the codomain
    functor:
    \begin{displaymath}
      \cod \from \functorcat{\lincat 2}{\cat B} \to \cat B,\,
      (X\overset f \to A) \mapsto A
    \end{displaymath}
    Cocartesian morphisms above $u$ relatively to $\cod$ are those
    commutative square of the form
    \begin{displaymath}
      \begin{tikzcd}
        X \dar["f"] \ar[r,equal] & X \dar["uf"] \\
        A \ar[r,"u"] & B
      \end{tikzcd}
    \end{displaymath}
    whereas cartesian morphisms above $u$ are the pullback squares
    along $u$. Hence $\cod$ is a Grothendieck bifibration whenever
    $\cat B$ admits pullbacks.

    If moreover $\cat B$ is a model category, then each fiber
    $\fiber \cod A \simeq \slice{\cat B} A$ inherits a model structure
    (namely an arrow is a fibration or a cofibration if it is such as
    an arrow of $\cat B$), and the total fibrations and cofibrations
    coincide with the one in the injective model structure on
    $\functorcat{\lincat 2}{\cat B}$: i.e.\ a cofibration is a
    commutative square with the top and bottom arrows being
    cofibrations in $\cat B$, whereas fibrations are those commutative
    squares
    \begin{displaymath}
      \begin{tikzcd}
        X \ar[dd,"f"] \ar[rr] \drar["h",densely dashed] & & Y \ar[dd,"g"] \\
        & \fiberprod A Y {u,g} \dlar \urar & \\
        A \ar[rr,"u"] & & B
      \end{tikzcd}
    \end{displaymath}
    where both $u$ and $h$ are fibrations in $\cat B$.
  \item Similarly, the total fibrations and cofibrations of the
    Grothendieck bifibration
    $\dom \from \functorcat{\lincat 2}{\cat B} \to \cat B$ over a
    model category $\cat B$ are exactly those of the projective model
    structure on $\functorcat{\lincat 2}{\cat B}$.
  \item In both \cite{stanculescu:bifib-model} and
    \cite{harpaz-prasma:grothendieck-model}, the authors prove a
    theorem similar to our, putting a model structure on the total
    category of a Grothendieck bifibration under specific
    hypothesis. In both case, fibrations and cofibrations of this
    model structure end up being the total ones. The following theorem
    encompasses in particular this two results.
  \end{enumerate}
\end{example}

\section{A Grothendieck construction for Quillen bifibrations}%
\label{sec:actual-thm}%
Now we have the tools to move on to the main goal of this paper, which
is to turn a Grothendieck bifibration $p \from \cat E \to \cat B$ into
a Quillen bifibration whenever both the basis category $\cat B$ and
every fiber $\fiber{\cat E}A \ (A\in\cat B)$ admit model structures in
such a way that all the pairs of adjoint push and pull functors
between fibers are ``homotopically well-behaved''. To be more precise,
we now suppose $\cat B$ to be equipped with a model structure
$(\class C,\class W,\class F)$, and each fiber $\fiber {\cat E} A$
($A\in \cat B$) to be equipped with a model structure
$(\class C_A,\class W_A,\class F_A)$. We also make the following
\showcase{fundamental assumption}:
\begin{displaymath}
  \tag{\quillenpushpull}%
  \label{hyp:quillen-push-pull}%
  \parbox{\dimexpr\linewidth-7em}{%
    \strut%
    For all $u$ in $\cat B$, the adjoint pair $(\push u,\pull u)$
    is a Quillen adjunction.
    \strut%
  }%
\end{displaymath}
We defined in definition~\ref{def:total-model} notions of total
cofibrations and total fibrations, as well as their acyclic
counterparts. These are reminiscent of what happens with Quillen
bifibrations, but they can be defined for any Grothendieck bifibration
whose basis and fibers have model structures. We must insist that in
that framework, {\em total cofibrations} and {\em total fibrations}
are only names, and by no means are they giving the total category
$\cat E$ a model structure. Indeed, the goal of this section, and to
some extent even the goal of this paper, is to provide a complete
characterization, under hypothesis~\eqref{hyp:quillen-push-pull}, of
the Grothendieck bifibrations $p\from \cat E \to \cat B$ for which the
total cofibrations and total fibrations make $p$ into a Quillen
bifibration. For the rest of this section, we shall denote
$\totalcof{\cat E}$, $\totalfib{\cat E}$, $\totalacof{\cat E}$ and
$\totalafib{\cat E}$ for the respective classes of total cofibrations,
total fibrations, total acyclic cofibrations, and total acyclic
fibrations, that is:
\begin{displaymath}
  \begin{aligned}
    \totalcof{\cat E} &= \{f \from X \to Y \in \cat E : p(f) \in
    \class C, \pushfact f \in \class C_{pY} \},
    \\
    \totalfib{\cat E} &= \{f \from X \to Y \in \cat E : p(f) \in
    \class F, \pullfact f \in \class F_{pX} \},
    \\
    \totalacof{\cat E} &= \{f \from X \to Y \in \cat E : p(f) \in
    \class W \cap \class C, \pushfact f \in \class W_{pY} \cap \class
    C_{pY} \},
    \\
    \totalafib{\cat E} &= \{f \from X \to Y \in \cat E : p(f) \in
    \class W \cap \class F, \pullfact f \in \class W_{pX} \cap \class
    F_{pX} \}
  \end{aligned}
\end{displaymath}

\subsection{Main theorem}

In order to state the theorem correctly, we will need some
vocabulary. Recall that the {\em mate}
$\mu \colon \push {u'} \pull v \to \pull {v'} \push u$ associated to a
commutative square of $\cat B$
\begin{displaymath}
  \begin{tikzcd}
    A \ar[d,"u'" swap] \ar[r,"v"] & C \ar[d,"u"] \\
    C' \ar[r,"v'" swap] & B
  \end{tikzcd}
\end{displaymath}
is the natural transformation constructed at point
$Z \in \fiber {\cat E} C$ in two steps as follow: the composite
\begin{displaymath}
  \pull v Z \to Z \to \push u Z 
\end{displaymath}
which is above $uv$, factors through the cartesian arrow
$\cart {v'} {\push u Z} \from \pull {v'} \push u Z \to \push u Z$
(because $v'u' = uv$) into a morphism
$\pull v Z \to \pull {v'} {\push u Z}$ above $u'$, which in turn
factors through the cocartesian arrow
$\cocart {u'} {\push v Z} \from \pull {u'} \push v Z \to \pull {v'}
\push u Z$ giving rise to $\mu_Z$, as summarized in the diagram below.
\begin{displaymath}
  \begin{tikzcd}
    \pull v Z \ar[rr,"\cartz"] \dar["\cocartz"] & & Z \ar[dd,"\cocartz"] \\
    \push {u'} \pull v Z \drar[densely dashed,"\mu_Z"] & & \\
    & \pull {v'} \push u Z \rar["\cartz"] & \push u Z
  \end{tikzcd}
\end{displaymath}

\begin{definition}
  A commutative square of $\cat B$ is said to satifisfy the {\em
    homotopical Beck-Chevalley condition} if its mate is pointwise a
  weak equivalence.
\end{definition}
Consider then the following properties on the Grothendieck bifibration
$p$:
\begin{displaymath}
  \tag{\hbc} \label{hyp:hBC}%
  \parbox{\dimexpr\linewidth-7em}{%
    \strut
    Every commutative square of $\cat B$ of the form
    \begin{displaymath}
      \begin{tikzcd}[ampersand replacement=\&]
        A \ar[d,"u'" swap] \ar[r,"v"] \& C \ar[d,"u"] \\
        C' \ar[r,"v'" swap] \& B
      \end{tikzcd}
      \qquad
      \begin{aligned}
        u,u' &\in \class C \cap \class W,
        \\
        v,v' &\in \class F \cap \class W
      \end{aligned}
    \end{displaymath}
    satisfies the homotopical Beck-Chevalley condition.
    \strut
  }%
\end{displaymath}
and
\begin{displaymath}
  \tag{\weakcons} \label{hyp:weak-conservative}%
  \parbox{\dimexpr\linewidth-5em}{%
    \strut
    The functors $\push u$ and $\pull v$ are homotopically conservative
    whenever $u$ is an acyclic cofibration and $v$ an
    acyclic fibration.
    \strut
  }%
\end{displaymath}
The theorem states that this is exactly what it takes to make the
names ``total cofibrations'' and ``total fibrations'' legitimate, and
to turn $p \from \cat E \to \cat B$ into a Quillen bifibration.
\begin{theorem}%
  \label{thm:main}%
  Under hypothesis \eqref{hyp:quillen-push-pull}, the total category
  $\cat E$ admits a model structure with $\totalcof{\cat E}$ and
  $\totalfib{\cat E}$ as cofibrations and fibrations respectively if
  and only if properties \eqref{hyp:hBC} and
  \eqref{hyp:weak-conservative} are satisfied.

  In that case, the functor $p \from \cat E \to \cat B$ is a Quillen
  bifibration.
\end{theorem}

The proof begin with a very candid remark that we promote as a
proposition because we shall use it several times in the rest of the
proof.
\begin{proposition}
  \label{prop:total-intertwined-wfs}%
  $(\totalacof{\cat E},\totalfib{\cat E})$ and
  $(\totalcof{\cat E},\totalafib{\cat E})$ are intertwined weak
  factorization systems.
\end{proposition}
\begin{proof}
  Obviously $\totalacof{\cat E} \subseteq \totalcof{\cat E}$ and
  $\totalafib{\cat E} \subseteq \totalfib{\cat E}$. Independently, a
  direct application of lemma \ref{lem:stan-lemma} shows that
  $(\totalacof{\cat E},\totalfib{\cat E})$ and
  $(\totalcof{\cat E},\totalafib{\cat E})$ are both weak factorization
  systems on $\cat E$.
\end{proof}
The strategy to prove theorem \ref{thm:main} then goes as follow:
\begin{itemize}
\item first we will show the necessity of conditions~\eqref{hyp:hBC}
  and~\eqref{hyp:weak-conservative}: if $\totalcof{\cat E}$ and
  $\totalfib{\cat E}$ are the cofibrations and fibrations of a model
  structure on $\cat E$, then hypothesis~\eqref{hyp:hBC}
  and~\eqref{hyp:weak-conservative} are met,
\item next, the harder part is the sufficiency: because of proposition
  \ref{prop:total-intertwined-wfs}, it is enough to show that the
  induced class
  $\totalweak{\cat E} = \totalafib{\cat E} \circ \totalacof{\cat E} $
  of {\em total weak equivalences} has the 2-out-of-3 property to
  conclude through corollary \ref{cor:intertwined-model}.
\end{itemize}

\subsection{Proof, part I: necessity}
\label{subsec:necessity-main}

In all this section, we suppose that $\totalcof{\cat E}$ and
$\totalfib{\cat E}$ provide respectively the cofibrations and
fibrations of a model structure on the total category $\cat E$. We
will denote $\totalweak{\cat E}$ the corresponding class of weak
equivalences.

First, we prove a technical lemma, directly following from proposition
\ref{prop:total-intertwined-wfs}, that will be extensively used in the
following. Informally, it states that the name given to the members of
$\totalacof{\cat E}$ and $\totalafib{\cat E}$ are not foolish.
\begin{lemma}
  \label{lem:total-acyclic-are-acyclic}%
  $\totalacof{\cat E} = \totalweak{\cat E} \cap \totalcof{\cat E}$ and
  $\totalafib{\cat E} = \totalweak{\cat E} \cap \totalfib{\cat E}$.
\end{lemma}
\begin{proof}
  By proposition \ref{prop:total-intertwined-wfs}, we know that both
  $(\totalacof{\cat E},\totalfib{\cat E})$ and
  $(\totalweak{\cat E} \cap \totalcof{\cat E},\totalfib{\cat E})$ are
  weak factorization systems with the same class of right maps, hence
  their class of left maps should coincide. Similarly the weak
  factorization systems $(\totalcof{\cat E},\totalafib{\cat E})$ and
  $(\totalcof{\cat E}, \totalweak{\cat E} \cap \totalfib{\cat E})$
  have the same class of left maps, hence their class of right maps
  coincide.
\end{proof}

\begin{corollary}
  \label{cor:inclusion-weak-conservative}%
  For any object $A$ of $\cat B$, the inclusion functor
  $\fiber{\cat E}A \to \cat E$ is homotopically conservative.
\end{corollary}
\begin{proof}
  The preservation of weak equivalences comes from the fact that
  acyclic cofibrations and acyclic fibrations of $\fiber{\cat E}A$ are
  elements of $\totalacof{\cat E}$ and $\totalafib{\cat E}$
  respectively. Thus, by lemma \ref{lem:total-acyclic-are-acyclic},
  they are elements of $\totalweak{\cat E}$.

  Conversely, suppose that $f$ is a map of $\fiber{\cat E} A$ which is
  a weak equivalence of $\cat E$. We want to show that $f$ is a weak
  equivalence of the fiber $\fiber{\cat E}A$. The map $f$ factors in
  the fiber $\fiber{\cat E}A$ as $f = qj$ where
  $j \in \class C_A \cap \class W_A$ and $q \in \class F_A$. We just
  need to show that $q \in \class W_A$. By lemma
  \ref{lem:total-acyclic-are-acyclic}, $j$ is also a weak equivalence
  of $\cat E$. By the 2-out-of-3 property of $\class W_{\cat E}$, the
  map $q$ is a weak equivalence of $\cat E$. As a fibration of
  $\fiber{\cat E}A$, $q$ is also a fibration of $\cat E$. This
  establishes that $q$ is an acyclic fibration of $\cat E$. By lemma
  \ref{lem:total-acyclic-are-acyclic}, $q$ is thus an element of
  $\totalafib{\cat E}$. This conludes the proof that $q = \pullfact q$
  is an acyclic fibration, and thus a weak equivalence, in the fiber
  $\fiber{\cat E}A$.
\end{proof}

\begin{proposition}[Property (\weakcons)]
  If $j\from A \to B$ in an acyclic cofibration in $\cat B$, then
  $\push j \from \fiber{\cat E} A \to \fiber{\cat E} B$ is
  homotopically conservative.
  
  If $q\from A \to B$ in an acyclic fibration in $\cat B$, then
  $\pull q \from \fiber{\cat E} B \to \fiber{\cat E} A$ is
  homotopically conservative.
\end{proposition}
\begin{proof}
  We only prove the first part of the proposition, as the second one
  is dual. Recall that the image $\push j (f)$ of a map
  $f\from X \to Y$ of $\fiber{\cat E} A$ is computed as the unique
  morphism of $\fiber{\cat E} B$ making the following square commute:
  \begin{displaymath}
    \begin{tikzcd}
      X \rar \dar["f",swap] & \push j X \dar["\push j(f)"] \\
      Y \rar & \push j Y
    \end{tikzcd}
  \end{displaymath}
  The horizontal morphisms in the diagram are cocartesian above the
  acyclic cofibration $j$. As such they are elements
  $\totalacof{\cat E}$, and thus weak equivalence in $\cat E$ by lemma
  \ref{lem:total-acyclic-are-acyclic}. By the 2-out-of-3 property of
  $\totalweak {\cat E}$, $f$ is a weak equivalence in $\cat E$ if and
  only if $\push j (f)$ is one also in $\cat E$. Corollary
  \ref{cor:inclusion-weak-conservative} allows then to conclude: $f$
  is a weak equivalence in the fiber $\fiber{\cat E} A$ if and only if
  $\push j (f)$ is one in the fiber $\fiber{\cat E} B$.
\end{proof}

\begin{proposition}[Property (\hbc)]
  Commutative squares of $\cat B$ of the form
  \begin{displaymath}
    \begin{tikzcd}
      A \rar["v"] \dar["u'",swap] & C \dar["u"] \\
      C' \rar["v'",swap] & B 
    \end{tikzcd}
    \quad u,u' \in \class C \cap \class W
    \quad v,v' \in \class F \cap \class W
  \end{displaymath}
  satisfy the homotopical Beck-Chevalley condition.
\end{proposition}
\begin{proof}
  Recall that for such a square in $\cat B$, the component of the mate
  $\mu \from \push {u'} \pull v \to \pull {v'} \push u$ at
  $Z \in \fiber{\cat E}C$ is defined as the unique map of
  $\fiber{\cat E}{Z'}$ making the following diagram commute:
  \begin{displaymath}
    \begin{tikzcd}
      \pull v Z \ar[rr,"\cartz"] \dar["\cocartz"] & & Z \ar[dd,"\cocartz"] \\
      \push {u'} \pull v Z \drar["\mu_Z"] & & \\
      & \pull {v'} \push u Z \rar["\cartz"] & \push u Z
    \end{tikzcd}
  \end{displaymath}
  Arrows labelled $\cartz$ and $\cocartz$ are respectively cartesian
  above acyclic fibrations and cocartesian above acyclic cofibrations,
  hence weak equivalences of $\cat E$ by lemma
  \ref{lem:total-acyclic-are-acyclic}. By applying the 2-out-of-3
  property of $\totalweak{\cat E}$ three times in a row, we conclude
  that the fiber map $\mu_Z$ is a weak equivalence of $\cat E$, hence
  also of $\fiber{\cat E}{C'}$ by corollary
  \ref{cor:inclusion-weak-conservative}.
\end{proof}

\subsection{Proof, part II: sufficiency}
\label{subsec:sufficiency}

We have established the necessity of \eqref{hyp:hBC} and
\eqref{hyp:weak-conservative} in theorem \ref{thm:main}. We now prove
the sufficiency of these conditions. This is the hard part of the
proof. Recall that every fiber $\fiber{\cat E}A$ of the Grothendieck
bifibration $p \from \cat E \to \cat B$ is equipped with a model
structure in such a way that \eqref{hyp:quillen-push-pull} is
satisfied. From now on, we make the additional assumptions that
\eqref{hyp:hBC} and \eqref{hyp:weak-conservative} are satisfied.

We will use the notation
$\totalweak{\cat E} = \totalafib{\cat E} \circ \totalacof{\cat E}$ the
class of maps that can be written as a total acyclic cofibration
postcomposed with a total acyclic fibration. The overall goal of this
section is to prove that
\begin{claim*}
  $(\totalcof{\cat E},\totalweak{\cat E},\totalfib{\cat E})$ defines a
  model structure on the total category $\cat E$.
\end{claim*}
By proposition \ref{prop:total-intertwined-wfs}, we already know that
$(\totalacof{\cat E},\totalfib{\cat E})$ and
$(\totalcof{\cat E},\totalafib{\cat E})$ are intertwined weak
factorization systems. From this follows that, by corollary
\ref{cor:intertwined-model}, we only need to show that the class
$\totalweak{\cat E}$ of \define{total weak equivalences} satisfies the
2-out-of-3 property.

\medskip

A first step is to get a better understanding of the total weak
equivalences. For $f \from X \to Y$ in $\cat E$ such that $p(f) = vu$
for two composable morphisms $u\from pX \to C$ and $v\from C \to pY$
of $\cat B$, there is a unique morphism inside the fiber
$\fiber{\cat E} C$
\begin{displaymath}
  \middlefact f u v \from \push u X \to \pull v Y
\end{displaymath}
such that $f = \cart v Y \circ \middlefact f u v \circ \cocart u
X$. This morphism $\middlefact f u v$ can be constructed as
$\pullfact k$ where $k$ is the unique morphism above $v$ factorizing
$f$ through $\cocart u X$; or equivalently as $\pushfact \ell$ where
$\ell$ is the unique morphism above $u$ factorizing $f$ through
$\cart v Y$. This is summed up in the following commutative diagram:
\begin{displaymath}
  \begin{tikzcd}[row sep=large, column sep=huge]
    X \rar["\cocartz"] \drar[densely dashed,"\ell" swap] & \push u X
    \dar["\pushfact \ell = \middlefact f u v = \pullfact k"
    description] \drar[densely dashed,"k"] &
    \\
    & \pull v Y \rar["\cartz" swap] & Y
  \end{tikzcd}
\end{displaymath}

Notice that, in particular, a morphism $f$ of $\totalweak{\cat E}$ is
exactly a morphism of $\cat E$ for which \showcase{there exists} a
factorization $p(f) = qj$ with $j \in \class W \cap \class C$ and
$q \in \class W \cap \class F$ such that $\middlefact f j q$ is a weak
equivalence in the corresponding fiber. We shall strive to show that,
under our hypothesis \eqref{hyp:weak-conservative} and
\eqref{hyp:hBC}, a morphism $f$ of $\totalweak{\cat E}$ satisfies the
same property that $\middlefact f j q$ is a weak equivalence
\showcase{for all} such factorization $p(f) = qj$. This is the contain
of proposition \ref{prop:exists-forall}. We start by showing the
property in the particular case where $p(f)$ is an acyclic cofibration
(lemma \ref{lem:exists-forall-above-acof}) or an acyclic fibration
(lemma \ref{lem:exists-forall-above-afib}).
\begin{lemma}
  \label{lem:exists-forall-above-acof}%
  Suppose that $f\from X \to Y$ is a morphism of $\cat E$ such that
  $p(f)$ is an acyclic cofibration in $\cat B$. If $p(f) = qj$ with
  $q \in \class W \cap \class C$ and $j \in \class W \cap \class F$,
  then $\pushfact f$ is a weak equivalence if and only if
  $\middlefact f j q$ is a weak equivalence.
\end{lemma}
\begin{proof}
  Since $p(f) = qj$, lemma \ref{lem:pseudo-push-iso} provides an
  isomorphism $\phi$ in the fiber $\fiber{\cat E}{pY}$ such that
  $\pushfact f = \tilde {\pushfact f} \phi$, where
  $\tilde {\pushfact f}$ is the morphism obtained by pushing in two
  steps:
  \begin{displaymath}
    \begin{tikzcd}
      X \rar["\cocartz"] \ar[rrr,bend left,"\cocartz"]
      \ar[drr,"f"{swap}] & \push j X \rar["\cocartz"]
      & \push q \push j X \dar["\tilde{\pushfact f}"] & \push
      {(qj)} X \dlar["\pushfact f"] \lar["\simeq","\phi"{swap}]
      \\
      & & Y &
    \end{tikzcd}
  \end{displaymath}
  By definition, $\middlefact f j q$ is the image of
  $\tilde {\pushfact f}$ under the natural bijection
  $\hom[\fiber{\cat E}{pY}]{\push q \push j X}{Y} \overset \simeq\to
  \hom[\fiber{\cat E}{pX}]{\push j X}{\pull q Y}$. So it can be
  written
  $\middlefact f j q = \pull q (\tilde {\pushfact f}) \circ
  \eta_{\push j X}$ using the unit $\eta$ of the adjunction
  $(\push q,\pull q)$. We can now complete the previous diagram as
  follow:
  \begin{displaymath}
    \begin{tikzcd}[row sep=large,column sep=large]
      X \rar["\cocartz"] \ar[rrr,bend left=20,"\cocartz"]
      \ar[drr,lightgray,"f"{very near start,swap}] & \push j X
      \dar[crossing over,"\eta"{fill=white}] \rar["\cocartz"]
      \ar[dd,crossing over,bend right=60,"\middlefact fjq" swap]&
      \push q \push j X \dar["\tilde{\pushfact f}"] & \push {(qj)} X
      \dlar["\pushfact f"] \lar["\simeq","\phi"{swap}]
      \\
      & \pull q \push q \push j X \dar["\pull q(\tilde{\pushfact f})"]
      \urar[crossing over,"\cartz" swap] & Y &
      \\
      & \pull q Y \urar["\cartz" swap] & &
    \end{tikzcd}
  \end{displaymath}
  Proving that $\eta_{\push j X}$ is a weak equivalence is then enough
  to conclude: in that case $\middlefact f j q$ is an weak equivalence
  if and only if $\pull q(\tilde {\pushfact f})$ is such by the
  two-of-three property ; $\pull q(\tilde {\pushfact f})$ is a weak
  equivalence if and only if $\tilde {\pushfact f}$ is a weak
  equivalence in $\fiber{\cat E}{pY}$ by \eqref{hyp:weak-conservative}
  ; and finally $\tilde {\pushfact f}$ is a weak equivalence if an
  only if $\pushfact f$ is such because they are isomorphic as arrows
  in $\fiber{\cat E}{pY}$.

  So it remains to show that $\eta_{\push j X}$ is a weak equivalence
  in its fiber. Since $p(f) = qj$, the following square commutes in
  $\cat B$:
  \begin{displaymath}
    \begin{tikzcd}
      pX \rar[equal,"\id {pX}"] \dar["j" swap]& pX \dar["qj"] \\
      C \rar["q" swap] & pY
    \end{tikzcd}
  \end{displaymath}
  This is a square of the correct form to apply \eqref{hyp:hBC}: hence
  the associated mate at component $X$
  \begin{displaymath}
    \mu_X \from \push j \pull {(\id {pX})} X \to \pull q \push {(qj)} X
  \end{displaymath}
  is a weak equivalence in the fiber $\fiber{\cat E} C$. Corollary
  \ref{cor:pseudo-iso-in-fiber} ensures that $\mu_X$ is isomorphic as
  arrow of $\fiber{\cat E}C$ to the unique fiber morphism that factors
  $\cart q {\push q \push j X}$ through $\cocart{q} {\push j X}$:
  \begin{displaymath}
    \begin{tikzcd}
      \push j X \rar["\cocartz"] \dar & \push q \push j X
      \\
      \pull q \push q \push j X \urar["\cartz" swap] &
    \end{tikzcd}
  \end{displaymath}
  This is exactly the definition of the unit $\eta$ at $\push j
  X$. Isomorphic morphisms being weak equivalences together,
  $\eta_{\push j X}$ is also acyclic in $\fiber{\cat E} C$.
\end{proof}

Of course, one gets the dual lemma by dualizing the proof that we let
for the reader to write down.
\begin{lemma}
  \label{lem:exists-forall-above-afib}%
  Let $f\from X \to Y$ a morphism of $\cat E$ such that $p(f)$ is an
  acyclic fibration in $\cat B$. If $p(f) = qj$ with
  $q \in \class W \cap \class C$ and $j \in \class W \cap \class F$,
  then $\pullfact f$ is a weak equivalence if and only if
  $\middlefact f j q$ is a weak equivalence.
\end{lemma}

We shall now prove the key proposition of this section.
\begin{proposition}
  \label{prop:exists-forall}%
  Let $f \from X \to Y$ in $\cat E$. If $p(f) = qj = q'j'$ for some
  $j,j \in \class W \cap \class C$ and
  $q,q' \in \class W \cap \class F$, then $\middlefact f j q$ is a
  weak equivalence if and only if $\middlefact f {j'} {q'}$ is a weak
  equivalence.
\end{proposition}
\begin{proof}
  By hypothesis the following square commutes in $\cat B$:
  \begin{displaymath}
    \begin{tikzcd}
      pX \rar["j'"] \dar["j" swap] & C' \dar["q'"] \\
      C \rar["q" swap] & pY
    \end{tikzcd}
  \end{displaymath}
  Since $j$ is an acyclic cofibration and $q'$ a (acyclic) fibration,
  there is a filler $h \from C \to C'$ of the previous square, that is
  a weak equivalence by the 2-out-of-3 property. Hence it can be
  factored $h=h_fh_c$ as an acyclic cofibration followed by an acyclic
  fibration in $\cat B$. Write $j'' = h_cj$ and $q'' = q'h_f$ which
  are respectively an acyclic cofibration and an acyclic fibration as
  composite of such, and produce a new factorization of
  $p(f) = q''j''$.
  \begin{displaymath}
    \begin{tikzcd}
      pX \ar[rr,"j'"] \ar[dd,"j" swap] \drar["j''"] & & C'
      \ar[dd,"q'"]
      \\
      & C'' \urar["h_f"] \drar["q''"] &
      \\
      C \urar["h_c"] \ar[rr,"q" swap] & & pY
    \end{tikzcd}
  \end{displaymath}
  Write $r$ for the composite
  $\middlefact f {j'} {q'} \circ \cocart X {j'} \from X \to \push {j'}
  X \to \pull {q'} Y$. Then $r$ is above the acyclic cofibration
  $j' = h_fj''$ and lemma \ref{lem:exists-forall-above-acof} can be
  applied: $\pushfact r$ is a weak equivalence in $\fiber{\cat E}{C'}$
  if and only if
  $\middlefact r {j''} {h_f} \from \push {j''} X \to \pull {(h_f)}
  \pull{q'} Y$ is a weak equivalence in $\fiber{\cat E}{C''}$. And by
  very definition $\pushfact r = \middlefact {j'} f {q'}$. So
  $\middlefact {j'} f {q'}$ is a weak equivalence in
  $\fiber{\cat E}{C'}$ if and only if $\middlefact r {j''} {h_f}$ is
  such in $\fiber{\cat E}{C''}$.
  
  Similarly write $s$ for the composite
  $\cart q Y \circ \middlefact f j q \from \push j X \to \pull q Y\to
  Y$. Then $s$ is above the acyclic fibration $q = q''h_c$ and lemma
  \ref{lem:exists-forall-above-afib} can be applied: $\pullfact s$ is
  a weak equivalence in $\fiber{\cat E}{C}$ if and only if
  $\middlefact s {h_c} {q''} \from \push {(h_c)} \push j X \to \pull
  {q''} Y $ is a weak equivalence (in $\fiber{\cat E}{C''}$). And by
  very definition $\pullfact s = \middlefact {j} f {q}$. So
  $\middlefact {j} f {q}$ is a weak equivalence in $\fiber{\cat E}{C}$
  if and only if $\middlefact s {h_c} {q''}$ is such in
  $\fiber{\cat E}{C''}$.

  Now recall that $j'' = h_cj$ and $q'' = q'h_f$. By lemmas
  \ref{lem:pseudo-push-iso} and \ref{lem:pseudo-pull-iso}, there
  exists isomorphisms $\push {j''} X \simeq \push{h_c}\push{j} X$ and
  $\pull{q''}Y \simeq \pull {h_f} \pull {q'} Y$ in fiber
  $\fiber{\cat E}{C''}$ making the following commute :
  \begin{displaymath}
    \begin{tikzcd}[row sep = large]%
      \push {h_c} \push {j} X \drar["\middlefact s {h_c} {q''}" swap]
      \rar["\simeq"] & \push{j''} X \dar["\middlefact f {j''} {q''}"
      description] \drar["\middlefact r {j''} {h_f}"] &
      \\
      & \pull{q''} Y \rar["\simeq"] & \pull {h_f} \pull {q'} Y
    \end{tikzcd}
  \end{displaymath}
  In particular, the morphisms $\middlefact r {j''} {h_f}$ and
  $\middlefact s {h_c} {q''}$ are weak equivalences together. We
  conclude the argument: $\middlefact f {j'} {q'}$ is a weak
  equivalence in $\fiber{\cat E}{C'}$ if and only if
  $\middlefact r {j''} {h_f}$ is such in $\fiber{\cat E}{C''}$ if and
  only if $\middlefact s {h_c} {q''}$ is such in $\fiber{\cat E}{C''}$
  if and only if $\middlefact f j q$ is a weak equivalence in
  $\fiber{\cat E}C$.
\end{proof}
The previous result allow the following ``trick'': to prove that a map
$f$ of $\cat E$ is in $\totalweak{\cat E}$, you just need to find
\showcase{some} factorization $p(f) = qj$ as an acyclic cofibration
followed by an acyclic fibration such that $\middlefact f j q$ is
acyclic inside its fiber (this is just the definition of
$\totalweak{\cat E}$ after all); but if given that
$f \in \totalweak{\cat E}$, you can use that $\middlefact f j q$ is a
weak equivalence for \showcase{every} admissible factorization of
$p(f)$!

\medskip%

We shall use that extensively in the proof of the two-out-of-three
property for $\totalweak{\cat E}$. This will conclude the proof of
sufficiency in theorem \ref{thm:main}.
\begin{proposition}
  The class $\totalweak{\cat E}$ has the 2-out-of-3 property.
\end{proposition}
\begin{proof}
  We suppose given a commutative triangle $h = gf$ in the total
  category $\cat E$, and we proceed by case analysis.

  First case: suppose that $f,g \in \totalweak{\cat E}$, and we want
  to show that $h \in \totalweak{\cat E}$. Since $f$ and $g$ are
  elements of $\totalweak {\cat E}$, there exists a pair of
  factorizations $p(f) = qj$ and $p(g) = q'j'$ with $j,j'$ acyclic
  cofibrations and $q,q'$ acyclic fibrations of $\cat B$ such that
  both $\middlefact f j q$ and $\middlefact g {q'} {j'}$ are weak
  equivalences in their respective fibers. The weak equivalence $j'q$
  of $\cat B$ can be factorized as $q''j''$ with $j''$ acyclic
  cofibration and $q''$ acyclic fibration. We write $i = j''j$ and
  $r = q'q''$ and we notice that $p(h) = ri$, as depicted below.
  \begin{equation}
    \label{eq:decompose-pf-and-pg-acyclic}%
    \begin{tikzcd}
      pX \ar[rr,bend left=40, "i"] \rar["j"] \drar[lightgray,"p(f)" swap]
      \ar[ddrr,lightgray,bend left,"p(h)"] & A
      \dar["q"{fill=white},crossing over] \rar["j''"] & C \dar["q''"]
      \ar[dd,bend left=50,"r"]
      \\
      & pY \rar["j'"{fill=white},crossing over] \drar[lightgray,"p(g)"
      swap] & B \dar["q'"]
      \\
      & & pZ
    \end{tikzcd}
  \end{equation}
  Since $i$ is an acyclic cofibration and $r$ is an acyclic fibration,
  it is enough to show that
  $\middlefact h i r \from \push i X \to \pull r Y$ is a weak
  equivalence in $\fiber{\cat E}C$ in order to conclude that
  $h \in \totalweak{\cat E}$. Since $i = j''j$ and $r=q'q''$,
  corollary \ref{cor:pseudo-iso-in-fiber} states that it is equivalent
  to show that the isomorphic arrow
  $\tilde h \from \push {j''} \push j X \to \pull {q''} \pull {q'}$ is
  a weak equivalence, where $\tilde h$ is defined as the unique arrow
  in fiber $\fiber{\cat E}C$ making the following commute:
  \begin{displaymath}
    \begin{tikzcd}
      X \rar["\cocartz"] \ar[drrrr,"h" near end,rounded corners,to
      path={(\tikztostart.south) |- ([yshift=1em]\tikztotarget.north)
        -- (\tikztotarget)}] & \push j X \rar["\cocartz"] & \push
      {j''} {\push j X} \ar[d,"\tilde h"{fill=white},crossing over] &
      &
      \\
      & & \pull{q''} \pull{q'} Z \rar["\cartz" swap] & \pull {q'} Z
      \rar["\cartz" swap] & Z
    \end{tikzcd}
  \end{displaymath}
  Since $h=gf$, such an arrow $\tilde h$ is given by the composite
  \begin{displaymath}
    \begin{tikzcd}[column sep=huge]
      \push {j''} \push {j'} X \rar[" \push {j''}(\middlefact f {j}
      {q})"] & \push {j''} \pull q Y \rar["\mu_Y"] &
      \pull{q''}\push{j'}Y \rar["\pull {q''}(\middlefact g {j'}
      {q'})"] & \pull {q''} \pull {q'} Z
    \end{tikzcd}
  \end{displaymath}
  where $\mu_Y$ is the component at $Y$ of the mate
  $\mu \from \push{j''}\pull q \to\pull{q''}\push{j'}$ of the
  commutative square $q''j'' = j'q$ of $\cat B$ (see diagram
  \eqref{eq:decompose-pf-and-pg-acyclic} above).
  \begin{displaymath}
    \begin{tikzcd}%
      X \ar[rrrr,lightgray,"f",bend left=20] \rar["\cocartz"] & \push
      j X \rar["\cocartz"] \dar["\middlefact f j q"]& \push {j''}
      {\push j X} & & \color{lightgray} Y \ar[ddd,lightgray,"g",bend
      left=20]
      \\
      & \pull q Y \ar[urrr,"\cartz" near end,lightgray,bend left=10]
      \rar["\cocartz"{fill=white}] & \push{j''} \pull q Y
      \dar["\mu_Y"] \uar["\push {j''} (\middlefact f j
      q)"{fill=white,swap},crossing over,leftarrow] & &
      \\
      & & \pull{q''} \push{j'} Y \rar["\cartz" swap] \dar["\pull
      {q''}(\middlefact g {j'} {q'})" swap] & \push {j'} Y
      \dar["\middlefact g {j'} {q'}" swap]
      \ar[uur,lightgray,"\cocartz",leftarrow,bend left=10] &
      \\
      & & \pull{q''} \pull{q'} Z \rar["\cartz" swap] & \pull {q'} Z
      \rar["\cartz" swap] & Z
    \end{tikzcd}
  \end{displaymath}
  We can conclude that $\tilde h$ is a weak equivalence in
  $\fiber{\cat E}C$ because it is a composite of such. Indeed:
  \begin{itemize}
  \item hypothesis \eqref{hyp:hBC} can be applied to the square
    $q''j''=j'q$, and so $\mu_Y$ is a weak equivalence in
    $\fiber{\cat E}C$,
  \item and by hypothesis \eqref{hyp:weak-conservative}, the functors
    $\push{j''}$ and $\pull {q''}$ maps the weak equivalences
    $\middlefact f j q$ and $\middlefact g {j'} {q'}$ to weak
    equivalences in $\fiber{\cat E}C$.
  \end{itemize}

  Suppose now that $f$ and $h$ are in $\totalweak {\cat E}$ and we
  will show that $g$ also is. Since $p(f)$ and $p(h)$ are weak
  equivalences in $\cat B$, we can use the two-out-of-three property
  of $\class W$ to deduce that also $p(g)$ is. By hypothesis,
  $p(f) = qj$ with $j\in \class C \cap \class W$ and
  $q \in \class F \cap \class W$ and $\middlefact f j q$ a weak
  equivalence. Also write $p(g) = q'j'$ for some
  $j'\in \class C \cap \class W$ and $q' \in \class F \cap \class
  W$. We are done if we show that $\middlefact g {j'} {q'}$ is a weak
  equivalence. But in that situation, one can define $j'',q'',i,r,$
  and $\tilde h$ as before. So we end up with the same big diagram,
  except that this time $\push {j''} (\middlefact f j q)$, $\mu_Y$ and
  the composite $\tilde h$ are weak equivalences of
  $\fiber{\cat E} C$, yielding $\pull {q''}(\middlefact g {j'} {q'})$
  as a weak equivalence by the 2-out-of-3 property. But $\pull {q''}$
  being homotopically conservative by \eqref{hyp:weak-conservative},
  this shows that $\middlefact g {j'} {q'}$ is a weak equivalence in
  $\fiber{\cat E}B$.

  The last case, where $g$ and $h$ are in $\totalweak{\cat E}$ is
  strictly dual.
\end{proof}

\section{Illustrations}
\label{sec:examples}

Since the very start, our work is motivated by the idea that the Reedy
model structure can be reconstructed by applying a series of
Grothendieck constructions of model categories. The key observation is
that the notion of latching and matching functors define a bifibration
at each step of the construction of the model structure. We explain in
\ref{subsec:reedy} how the Reedy construction can be reunderstood from
our bifibrational point of view. In section \ref{subsec:gen-reedy}, we
describe how to adapt to express generalized Reedy constructions in a
similar fashion. In section \ref{subsec:versus-hp-rs}, we recall the
previous notions of bifibration of model categories appearing in the
literature and, although all of them are special cases of Quillen
bifibrations, we indicate why they do not fit the purpose.

\subsection{A bifibrational view on Reedy model structures}
\label{subsec:reedy}

Recall that a \define{Reedy category} is a small category $\cat R$
together with two subcategories $\cat R^+$ and $\cat R^-$ and a degree
function $d \from \ob{\cat R} \to \lambda$ for some ordinal $\lambda$
such that
\begin{itemize}
\item every morphism $f$ admits a unique factorization $f = f^+f^-$
  with $f^- \in \cat R^-$ and $f^+ \in \cat R^+$,
\item non-identity morphisms of $\cat R^+$ strictly raise the degree
  and those of $\cat R^-$ strictly lower it.
\end{itemize}
For such a Reedy category, let $\deginf{\cat R} \mu$ denote the full
subcategory spanned by objects of degree strictly less than $\mu$. In
particular, $\cat R = \deginf{\cat R}\lambda$. Remark also that every
$\deginf{\cat R}\mu$ inherits a structure of Reedy category from
$\cat R$.

We are interested in the structure of the category of diagrams of
shape $\cat R$ in a complete and cocomplete category $\cat C$. The
category $\cat C$ is in particular tensored and cotensored over
$\Set$, those being respectively given by
\begin{displaymath}
  S \tens C = \coprod_{s \in S} C,
  \qquad
  S \cotens C = \prod_{s \in S} C,
  \qquad S\in\Set, C\in \cat C.
\end{displaymath}
For every $r \in \cat R$ of degree $\mu$, a diagram
$X \from \deginf{\cat R}\mu \to \cat C$ induces two objects in
$\cat C$, called the \define{latching} and \define{matching} objects
of $X$ at $r$, and respectively defined as:
\begin{displaymath}
  \latching r X =
  \colimend{s \in \deginf{\cat R}\mu} \hom[\cat R] s r \tens X_s
  , \qquad
  \matching r X =
  \limend{s \in \deginf{\cat R}\mu} \hom[\cat R] r s \cotens X_s
\end{displaymath}
By abuse, we also denote $\latching r X$ and $\matching r X$ for the
latching and matching objects of the restriction to
$\deginf{\cat R}\mu$ of some
$X \from \deginf{\cat R}\kappa \to \cat C$ with $\kappa \geq \mu$. In
particular, when $\kappa = \lambda$, $X$ is a diagram of shape the
entire category $\cat R$ and we retrieve the textbook notion of
latching and matching objects (see for instance
\cite{hovey:model-cat}). Universal properties of limits and colimits
induce a family of canonical morphisms
$\alpha_r \from \latching r X \to \matching r X$, which can also be
understood in the following way. First, one notices that the two
functors defined as $\deginf{\cat R}{\mu+1} \to \cat C$
\begin{displaymath}
  r \mapsto \left\{
    \begin{aligned}
      X_r \ &\text{if $d(r) < \mu$}\\
      \latching r X \ &\text{if $d(r) = \mu$}\\
    \end{aligned}
  \right.
  ,\qquad
  r \mapsto \left\{
    \begin{aligned}
      X_r \ &\text{if $d(r) < \mu$}\\
      \matching r X \ &\text{if $d(r) = \mu$}\\
    \end{aligned}
  \right.
\end{displaymath}
are the skeleton and coskeleton $X$, which provide a left and a right
Kan extensions $X$ along the inclusion
$\deginf i \mu \from \deginf{\cat R}{\mu} \to \deginf{\cat
  R}{\mu+1}$. We will write these two functors $\latching \mu X$ and
$\matching \mu X$ respectively. The family of morphisms $\alpha_r$
then describes the unique natural transformation
$\alpha \from \latching \mu X \to \matching \mu X$ that restrict to
the identity on $\deginf {\cat R} \mu$.

The following property is, in our opinion, the key feature of Reedy
categories.
\begin{proposition}
  \label{prop:key-feature-reedy}%
  Extensions of a diagram $X \from \deginf{\cat R}\mu \to \cat C$ to
  $\deginf{\cat R}{\mu+1}$ are in one-to-one correspondence with
  families of factorizations of the $\alpha_r$'s
  \begin{displaymath}
    \left( \latching r X \to \bullet \to \matching r X \right)_{r \in \cat R,d(r) = \mu}
  \end{displaymath}
\end{proposition}
\begin{proof}
  One direction is easy. Every extension
  $\hat X \from \deginf{\cat R}{\mu+1} \to \cat C$ of $X$ produces
  such a family of factorizations, but it has nothing to do with the
  structure of Reedy category: for every $r$ of degree $\mu$ in
  $\cat R$, the functoriality of $\hat X$ ensures that there is a
  coherent family of morphisms $X_s = \hat X_s \to \hat X_r$ for each
  arrow $s \to r$, and symetrically a coherent family of morphisms
  $\hat X_r \to \hat X_{s'} = X_{s'}$ for each arrow $r \to s'$. Hence
  the factorization of $\alpha_r$ given by the universal properties of
  limits and colimits
  \begin{displaymath}
    \latching r X \to \hat X_r \to \matching r X
  \end{displaymath}

  The useful feature is the converse: when usually, to construct an
  extension of $X$, one should define images for arrows $r \to r'$
  between objects of degree $\mu$ in a functorial way, here every
  family automatically induces such arrows! This is a fortunate effect
  of the unique factorization property. Given factorizations
  $\latching r X \to X_r \to \matching r X$, one can define $X(f)$ for
  $f \from r \to r'$ as follow: factor $f = f^+f^-$ with
  $f^- \from r \to s$ lowering the degree and $f^+ \from s \to r'$
  raising it, so that in particular $s \in \deginf{\cat R}\mu$; $f^-$
  then gives rise to a canonical projection $\matching r X \to X_s$
  and $f^+$ to a canonical injection $X_s \to \latching {r'} X$; the
  wanted arrow $X(f)$ is given by the composite
  \begin{displaymath}
    X_r \to \matching r X \to X_s \to \latching {r'} X \to X_{r'}
  \end{displaymath}
  Well-definition and functoriality of the said extension are
  following from uniqueness in the factorization property of the Reedy
  category $\cat R$.
\end{proof}

\medskip

From now on, we fix a model category $\cat M$, that is a complete and
cocomplete category $\cat M$ with a model structure
$(\class C,\class W,\class F)$.
The motivation behind Kan's notion of Reedy categories is
to gives sufficient conditions on $\cat R$ to equip
$\functorcat{\cat R}{\cat M}$ with a model structure where weak
equivalences are pointwise.
\begin{definition}%
  Let $\cat R$ be Reedy. The \define{Reedy triple} on the
  functor category $\functorcat{\cat R}{\cat M}$ is the data of the
  three following classes
  \begin{itemize}
  \item Reedy cofibrations : those $f \from X \to Y$ such that for all
    $r \in\cat R$, the map
    $\latching r Y \sqcup_{\latching r X} X_r \to Y_r$ is a
    cofibration,
  \item Reedy weak equivalences : those $f \from X \to Y$ such that
    for $r \in \cat R$, $f_r \from X_r \to Y_r$ is a weak equivalence,
  \item Reedy fibrations : those $f \from X \to Y$ such that for all
    $r \in\cat R$, the map
    $X_r \to \matching r X \times_{\matching r Y} Y_r$ is a fibration.
  \end{itemize}
\end{definition}
Kan's theorem about Reedy categories, whose our main result gives a
slick proof, then states as follow: the Reedy triple makes
$\functorcat{\cat R}{\cat M}$ into a model category. A first reading
of this definition/theorem is quite astonishing: the distinguished
morphisms are defined through those latching and matching objects, and
it is not clear, apart from being driven by the proof, why we should
emphasize those construction that much. We shall say a word about that
later.

\begin{remark}
  \label{rem:restr-along-inj-isofib}%
  Before going into
  proposition~\ref{prop:reedy-restriction-bifibration} below, we need
  to make a quick remark about extensions of diagrams up to
  isomorphism. Suppose given a injective-on-objects functor
  $i \from \cat A \to \cat B$ between small categories and a category
  $\cat C$, then for every diagram $D \from \cat A \to \cat C$, every
  diagram $D' \from \cat B \to \cat C$ and every isomorphism
  $\alpha \from D\to D'i$, there exists a diagram
  $D''\from \cat B \to \cat C$ isomorphic to $D'$ such that $D''i = D$
  (and the isomorphism $\beta\from D'' \to D'$ can be chosen so that
  $\beta i = \alpha$). Informally it says that every ``up to
  isomorphism'' extension of $D$ can be rectified into a strict
  extension of $D$.

  Put formally, we are claiming that the restriction functor
  $\restr i \from \functorcat{\cat B}{\cat C} \to \functorcat{\cat
    A}{\cat C}$ is an isofibration. Although it can be shown easily by
  hand, we would like to present an alternate proof based on
  homotopical algebra. Taking a universe $\mathbb U$ big enough for
  $\cat C$ to be small relatively to $\mathbb U$, we can consider the
  folk model structure on the category $\Cat$ of $\mathbb U$-small
  categories. With its usual cartesian product, $\Cat$ is a closed
  monoidal model category in which every object is fibrant. It follows
  that $\functorcat{-}{\cat C}$ maps cofibrations to fibrations (see
  \cite[Remark 4.2.3]{hovey:model-cat}). Then, the
  injective-on-objects functor $i \from \cat A\to \cat B$ is a
  cofibration, so it is mapped to a fibration
  $\restr i \from \functorcat{\cat B}{\cat C} \to \functorcat{\cat
    A}{\cat C}$. Recall that fibrations in $\Cat$ are precisely the
  isofibrations and we obtain the result.
\end{remark}

\begin{proposition}
  \label{prop:reedy-restriction-bifibration}%
  Let $\cat R$ be Reedy. The restriction functor
  $\restr{\deginf i \mu} \from \functorcat{\deginf{\cat
      R}{\mu+1}}{\cat M} \to \functorcat{\deginf{\cat R}{\mu}}{\cat
    M}$ is a Grothendieck bifibration.
\end{proposition}
\begin{proof}
  The claim is that a morphism $f \from X \to Y$ is cartesian
  precisely when the following diagram is a pullback square:
  \begin{equation}
    \label{eq:cartesian-claim}
    \begin{tikzcd}[column sep=large]
      X \dar \rar["f"] & Y \dar
      \\
      \matching \mu p X \rar["\matching \mu p(f)" swap] & \matching
      \mu p Y
    \end{tikzcd}
  \end{equation}
  where the vertical arrows are the component at $X$ and $Y$ of the
  unit $\eta$ of the adjunction $(p,\matching \mu)$. Indeed, such a
  diagram is a pullback square if and only if the following square is
  a pullback for all $Z$:
  \begin{displaymath}
    \begin{tikzcd}[column sep=large]
      \hom[\functorcat{\deginf{\cat R}{\mu+1}}{\cat M}] Z X
      \dar["\eta_X \circ {-}"] \rar["f\circ{-}"] &
      \hom[\functorcat{\deginf{\cat R}{\mu+1}}{\cat M}] Z Y
      \dar["\eta_Y \circ {-}"]
      \\
      \hom[\functorcat{\deginf{\cat R}{\mu+1}}{\cat M}] Z {\matching
        \mu p X} \rar["\matching \mu p(f) \circ {-}" swap] &
      \hom[\functorcat{\deginf{\cat R}{\mu+1}}{\cat M}] Z {\matching
        \mu p Y}
    \end{tikzcd}
  \end{displaymath}
  We can take advantage of the adjunction $(p,\matching \mu)$ and its
  natural isomorphism
  \begin{displaymath}
    \phi_{Z,A} \from
    \hom[\functorcat{\deginf{\cat R}{\mu+1}}{\cat M}] Z {\matching \mu A}
    \simeq \hom[\functorcat{\deginf{\cat R}{\mu}}{\cat M}] {pZ} A
  \end{displaymath}
  As in any adjunction, this isomorphism is related to the unit by the
  following identity: for any $g \from Z \to X$,
  $p(g) = \phi(\eta_X g)$. So in the end, the square
  in~(\ref{eq:cartesian-claim}) is a pullback if and only if for every
  $Z$ the outer square of the following diagram is a pullback:
 \begin{displaymath}
   \begin{tikzcd}[column sep=large]%
     \hom[\functorcat{\deginf{\cat R}{\mu+1}}{\cat M}] Z X
     \dar["\eta_X \circ {-}"] \rar["f\circ{-}"]
     \ar[dd,"p"{swap},rounded corners,to path={--
       ([xshift=-2em]\tikztostart.west) --
       ([xshift=-2em]\tikztotarget.west) \tikztonodes --
       (\tikztotarget)}] & \hom[\functorcat{\deginf{\cat
         R}{\mu+1}}{\cat M}] Z Y \dar["\eta_Y \circ {-}"]
     \ar[dd,"p",rounded corners,to path={--
       ([xshift=2em]\tikztostart.east) --
       ([xshift=2em]\tikztotarget.east) \tikztonodes --
       (\tikztotarget)}]
     \\
     \hom[\functorcat{\deginf{\cat R}{\mu+1}}{\cat M}] Z {\matching
       \mu p X} \rar["\matching \mu p(f) \circ {-}" swap]
     \dar["\phi"{swap},"\rotatebox{-90}{$\simeq$}"] &
     \hom[\functorcat{\deginf{\cat R}{\mu+1}}{\cat M}] Z {\matching
       \mu p Y} \dar["\phi"{swap},"\rotatebox{-90}{$\simeq$}"]
     \\
     \hom[\functorcat{\deginf{\cat R}{\mu}}{\cat M}] {pZ} {pX}
     \rar["p(f) \circ {-}" swap] 
     & \hom[\functorcat{\deginf{\cat R}{\mu}}{\cat M}] {pZ} {pY}
   \end{tikzcd}
 \end{displaymath}
  This is exactly the definition of a cartesian morphism.
  Dually, we can prove that cocartesian morphisms are those
  $f \from X \to Y$ such that the following is a pushout square:
  \begin{displaymath}
    \begin{tikzcd}
      \latching \mu p X \dar \rar["\latching \mu p(f)"] & \latching \mu p
      \dar Y
      \\
      X \rar["f" swap] & Y
    \end{tikzcd}
  \end{displaymath}

  Now for $u \from A \to pY$ in
  $\functorcat{\deginf{\cat R}{\mu}}{\cat M}$, one should construct a
  cartesian morphism $f \from X \to Y$ above $u$. First notice that we
  constructed $\matching \mu$ in such a way that
  $p \matching \mu = \id{}$ (even more, the counit
  $p \matching \mu \to \id{}$ is the identity natural
  transformation). So $\matching \mu A$ is above $A$ and we could be
  tempted to take, for the wanted $f$, the morphism
  $\kappa \from \matching \mu A \times_{\matching \mu pY} Y \to Y$
  appearing in the following pullback square:
  \begin{equation}
    \label{eq:pullback-not-exactly-above}%
    \begin{tikzcd}
      \bullet \dar \rar["\kappa"] & Y \dar
      \\
      \matching \mu A \rar["\matching \mu u" swap] & \matching
      \mu p Y
    \end{tikzcd}
  \end{equation}
  But $\kappa$ is not necessarily above $u$. Indeed, as a right
  adjoint, $p$ preserves pullbacks. So we get that the
  following is a pullback in
  $\functorcat{\deginf{\cat R}\mu}{\cat M}$:
  \begin{displaymath}
    \begin{tikzcd}
      p(\bullet) \dar \rar["p(\kappa)"] & pY \dar["\id Y"]
      \\
      A \rar["u" swap] & pY
    \end{tikzcd}
  \end{displaymath}
  We certainly know another pullback square of the same diagram,
  namely
  \begin{displaymath}
    \begin{tikzcd}
      A \dar["\id A" swap] \rar["u"] & pY \dar["\id Y"]
      \\
      A \rar["u" swap] & pY
    \end{tikzcd}
  \end{displaymath}
  So, by universal property, we obtain an isomorphism
  $\alpha\from A \to p(\matching \mu A \times_{\matching \mu pY}
  Y)$. Now we summon remark~\ref{rem:restr-along-inj-isofib} to get an
  extension $X$ of $A$ and an isomorphism
  $\beta \from X \to \matching \mu A \times_{\matching \mu pY} Y$
  above $\alpha$. The wanted $f \from X \to Y$ is then just the
  composite $\kappa\beta$, which is cartesian because the outer square
  in the following is a pullback (as we
  chose~\eqref{eq:pullback-not-exactly-above} to be one):
  \begin{displaymath}
    \begin{tikzcd}
      X \drar["\beta"] \ar[dd] \ar[rr,"f"] & & Y \ar[dd]
      \\
      & \bullet \dlar \urar["\kappa"] & 
      \\
      \matching \mu A \ar[rr,"\matching \mu u" swap] & & \matching
      \mu p Y
    \end{tikzcd}
  \end{displaymath}
  The fact that the vertical map
  $X \to \matching \mu A = \matching \mu p X$ is indeed the unit
  $\eta$ of the adjunction at component $X$ comes directly from the
  fact that its image by $p$ is $\id A$. The existence of cocartesian
  morphism above any $u \from pX \to B$ is strictly dual, using this
  time the cocontinuity of $p$ as a left adjoint.
\end{proof}

\begin{remark}
  First, we should notice that proposition
  \ref{prop:key-feature-reedy} make the following multievaluation
  functor an equivalence:
  \begin{displaymath}
    \tag{I} \label{eq:fiber-multieval}%
    \fiber{\functorcat{\deginf{\cat R}{\mu+1}}{\cat M}} A \xrightarrow \sim
    \prod_{r \in \cat R, d(r) = \mu} \coslice{\slice{\cat M}{\matching r A}}{\latching r A}
  \end{displaymath}
  The notation
  $\coslice{\slice{\cat M}{\matching r A}}{\latching r A}$ is slightly
  abusive and means the coslice category of
  $\slice{\cat M}{\matching r A}$ by $\alpha_r$, or equivalently the
  slice category of $\coslice{\cat M}{\latching r A}$ by $\alpha_r$.

  Secondly, we can draw from the previous proof that for a morphism
  $f \from X \to Y$, the fiber morphisms $\pullfact f$ and
  $\pushfact f$ are, modulo identification \eqref{eq:fiber-multieval},
  the respective induced families defining the Reedy triple:
  \begin{displaymath}
    (X_r \to \matching r X \times_{\matching r Y} Y_r)_{r,d(r)=\mu}
    ,\qquad
    (X_r \sqcup_{\latching r X} \latching r Y \to Y_r)_{r,d(r)=\mu}
  \end{displaymath}
  So here it is: the reason behind those {\em a priori} mysterious
  morphisms, involving latching an matching, are nothing else but the
  witness of a hidden bifibrational structure. Putting this into light
  was a tremendous leap in our conceptual understanding of Reedy model
  structures and their generalizations.
\end{remark}

The following proposition is the induction step for successor ordinals
in the usual proof of the existence of Reedy model structures. Our
main theorem \ref{thm:main} allows a very smooth argument.
\begin{proposition}
  \label{prop:reedy-from-main-thm}%
  If the Reedy triple on $\functorcat{\deginf{\cat R}{\mu}}{\cat M}$
  forms a model structure, then it is also the case on
  $\functorcat{\deginf{\cat R}{\mu+1}}{\cat M}$.
\end{proposition}
\begin{proof}
  Our course, the goal is to use theorem \ref{thm:main} on the
  Grothendieck bifibration
  $\restr{\deginf i \mu} \from \functorcat{\deginf{\cat
      R}{\mu+1}}{\cat M} \to \functorcat{\deginf{\cat R}{\mu}}{\cat
    M}$. By hypothesis, the base
  $\functorcat{\deginf{\cat R}{\mu}}{\cat M}$ has a model structure
  given by the Reedy triple. Each fiber
  $\fiber{(\restr{\deginf i \mu})} A$ above a diagram $A$ is endowed,
  via identification \eqref{eq:fiber-multieval}, with the product
  model structure: indeed, if $\cat N$ is a model category, so is its
  slices $\slice{\cat N}N$ and coslices $\coslice{\cat N}N$
  categories, just defining a morphism to be a cofibration, a
  fibration or a weak equivalence if it is in $\cat N$; products of
  model categories are model categories by taking the pointwise
  defined structure. All in all, it means that the following makes the
  fiber $\fiber{(\restr{\deginf i \mu})} A$ into a model category: a
  fiber map $f \from X \to X'$ in $\fiber{(\restr{\deginf i \mu})} A$
  is a cofibration, a fibration or a weak equivalence if and only if
  $f_r \from X_r \to X'_r$ is one for every $r \in \cat R$ of degree
  $\mu$.

  Now the proof amounts to show that hypothesis
  \eqref{hyp:quillen-push-pull}, \eqref{hyp:weak-conservative} and
  \eqref{hyp:hBC} are satisfied in this framework. Let us first tackle
  \eqref{hyp:quillen-push-pull}. Suppose $u \from A \to B$ in
  $\functorcat{\deginf{\cat R}{\mu}}{\cat M}$ and $f \from Y \to Y'$ a
  fiber morphism at $B$. Then by definition of the cartesian morphisms
  in $\functorcat{\deginf{\cat R}{\mu+1}}{\cat M}$, $\pull u f$ is the
  unique map above $A$ making the following diagram commute for all
  $r$ of degree $\mu$:
  \begin{equation}
    \label{eq:pullback-reedy}%
    \begin{tikzcd}
      (\pull u Y)_r \rar \dar["(\pull u f)_r" swap] & Y_r \dar["f_r"]
      \\
      (\pull u Y')_r \rar \dar & Y'_r \dar
      \\
      \matching r A \rar["\matching r u" swap] & \matching r B
    \end{tikzcd}
  \end{equation}
  where the lower square and outer square are pullback diagrams. By
  the pasting lemma, so is the upper square. Hence $(\pull u f)_r$ is
  a pullback of $f_r$, and as such is a (acyclic) fibration whenever
  $f_r$ is one. This proves that $\pull u$ is right Quillen for any
  $u$, that is \eqref{hyp:quillen-push-pull}.

  Goals \eqref{hyp:weak-conservative} and \eqref{hyp:hBC} will be
  handle pretty much the same way one another and it lies on the
  following well know fact about Reedy model structures \cite[lemma
  15.3.9]{hirschhorn:loc}: for $r \in \cat R$ of degree $\mu$, the
  functor
  $\matching r \from \functorcat{\deginf {\cat R} \mu}{\cat M} \to
  \cat M$ preserves acyclic fibrations\footnote{Actually it is right
    Quillen, but we will not need that much here.}. This has a
  wonderful consequence: if $u$ is an acyclic fibration of
  $\functorcat{\deginf {\cat R} \mu}{\cat M}$, any pullback of
  $\matching r u$ is an acyclic fibration hence a weak equivalence. So
  the upper square of diagram \eqref{eq:pullback-reedy} has acyclic
  horizontal arrows. By the 2-out-of-3 property, $f_r$ on the right is
  a weak equivalence if an only if $(\pull u f)_r$ is one. This being
  true for each $r \in \cat R$ of degree $\mu$ makes $\pull u$
  homotopically conservative whenever $u$ is an acyclic
  fibration. This validates half of the property
  \eqref{hyp:weak-conservative}. The other half is proven dually,
  resting on the dual lemma: for any $r \in \cat R$ of degree $\mu$,
  the latching functor
  $\latching r \from \functorcat{\deginf {\cat R} \mu}{\cat M} \to
  \cat M$ preserves acyclic cofibrations; then deducing that pushouts
  of $\latching r u$ are weak equivalences whenever $u$ is an acyclic
  cofibration.

  It remains to show \eqref{hyp:hBC}. Everything is already in place
  and it is just a matter of expressing it. For a commutative square
  of $\functorcat{\deginf {\cat R} \mu}{\cat M}$
  \begin{displaymath}
    \begin{tikzcd}
      A \rar["v"] \dar["u'" swap] & C \dar["u"]
      \\
      C' \rar["v'" swap] & B
    \end{tikzcd}
  \end{displaymath}
  with $u,u'$ Reedy acyclic cofibrations and $v,v'$ Reedy acyclic
  fibrations, the mate at an extension $Z$ of $C$ is the unique fiber
  morphism
  $\nu_Z \from (\push {u'} \pull v Z) \to (\pull {v'} \push u Z)$
  making the following commute for every $r \in \cat R$ of degree
  $\mu$:
  \begin{displaymath}
    \begin{tikzcd}
      & \color{lightgray} \matching r A \ar[rr,"\matching r
      v",lightgray] & & \color{lightgray} \matching r C &
      \\
      \color{lightgray} \latching r A \dar["\latching r u'"
      swap,lightgray] \rar[lightgray] & (\pull v Z)_r \ar[rr] \dar
      \uar[lightgray] & & Z_r \ar[dd] \uar[lightgray] &
      \color{lightgray} \latching r C \lar[lightgray]
      \ar[dd,"\latching r v",lightgray]
      \\
      \color{lightgray} \latching r C' \rar[lightgray] & (\push {u'}
      \pull v Z)_r \drar["(\nu_Z)_r"] & & &
      \\
      & & (\pull {v'} \push u Z)_r \rar \dar[lightgray] & (\push u
      Z)_r \dar[lightgray] & \color{lightgray} \latching r B
      \lar[lightgray]
      \\
      & & \color{lightgray} \matching r C' \rar["\matching r v'"
      swap,lightgray] & \color{lightgray} \matching r B &
    \end{tikzcd}
  \end{displaymath}
  where grayscaled square are either pullbacks (when involving
  matching objects) or pushouts (when involving latching objects). So
  by the same argument as above, the horizontal and vertical arrows of
  the pentagone are weak equivalences, making the $r$-component of the
  mate $(\nu_Z)_r$ a weak equivalence also by the 2-out-of-3 property.

  Theorem \ref{thm:main} now applies, and yield a model structure on
  $\functorcat{\deginf{\cat R}{\mu+1}}{\cat M}$ which is readily the
  Reedy triple.
\end{proof}


\subsection{Notions of generalized Reedy categories}
\label{subsec:gen-reedy}

From time to time, people stumble accross almost Reedy categories and
build {\em ad hoc} workarounds to end up with a structure ``à la
Reedy''. The most popular such generalizations are probably Cisinski's
\cite{cisinski:thesis} and Berger-Moerdijk's
\cite{berger-moerdijk:gen-reedy}, allowing for non trivial
automorphisms. In \cite{shulman:gen-reedy}, Shulman establishes a
common framework for every such known generalization of Reedy
categories (including enriched ones, which go behind the scope of this
paper). Roughly put, Shulman defines {\em almost-Reedy categories} to
be those small categories $\cat C$ with a degree function on the
objects that satisfy the following property: taking $x$ of degree
$\mu$ and denoting $\deginf{\cat C}\mu$ the full subcategory of
$\cat C$ of objects of degree strictly less than $\mu$, and
${\cat C}_x$ the full subcategory of $\cat C$ spanned by
$\deginf{\cat C}\mu$ and $x$, then the diagram category
$\functorcat{\cat C_x}{\cat M}$ is obtained as the {\em bigluing} (to
be defined below) of two nicely behaved functors
$\functorcat{\deginf{\cat C}\mu}{\cat M} \to \cat M$, namely the
weighted colimit and weigthed limit functors, respectively weighted by
$\hom[\cat C] - x$ and $\hom[\cat C] x -$. In particular, usual Reedy
categories are recovered when realizing that the given formulas of
latching and matching objects are precisely these weighted colimits
and limits.

\medskip

In order to understand completely the generalization proposed in
\cite{shulman:gen-reedy}, we propose an alternative view on the Reedy
construction that we exposed in detail in the previous section. For
starter, here is a nice consequence of theorem~\ref{thm:main}:
\begin{lemma}
  \label{lem:quillen-bifib-by-pullback}%
  Suppose there is a strict pullback square of categories
  \begin{displaymath}
    \begin{tikzcd}
      \cat F \rar \dar["q"] & \cat E \dar["p"]
      \\
      \cat C \rar["F"] & \cat B
    \end{tikzcd}
  \end{displaymath}
  in which $\cat C$ has a model structure and $p$ is a Quillen
  bifibration. If
  \begin{enumerate}[label=(\roman*)]
  \item \label{item:weak-conservative-through-F}%
    $\push{F(u)}$ and $\pull{F(v)}$ are homotopically conservative
    whenever $u$ is an acyclic cofibration and $v$ an acyclic
    fibration in $\cat C$,
  \item \label{item:hBC-through-F}%
    $F$ maps squares of the form
    \begin{displaymath}
      \begin{tikzcd}
        A \rar["v"] \dar["u'"] & C \dar["u"]
        \\
        C' \rar["v'"] & B
      \end{tikzcd}
    \end{displaymath}
    with $u,u'$ acyclic cofibrations and $v,v'$ acyclic fibrations in
    $\cat C$ to squares in $\cat B$ that satisfy the homotopical
    Beck-Chevalley condition,
  \end{enumerate}
  then $q$ is also a Quillen bifibration.
\end{lemma}
\begin{proof}
  Denote $p' \from \cat B \to \Adj$ the pseudo functor
  $A \mapsto \fiber{\cat E}A$ associated to $p$. Then it is widely
  known that the pullback $q$ of $p$ along $F$ is the bifibration
  obtained by Grothendieck construction of the pseudo functor
  $p'F\from \cat C \to \Adj$. It has fiber
  $\fiber{\cat F}C = \fiber{\cat E}{FC}$ at $C \in\cat C$, which has
  a model structure; and for any $u \from C \to D$ in $\cat C$, the
  adjunction
  \begin{displaymath}
    \push u \from \fiber{\cat F}C \rightleftarrows
    \fiber{\cat F}D \cofrom \pull u
  \end{displaymath}
  is given by the pair $(\push {F(u)}, \pull {F(u)})$ defined by
  $p$. Hence theorem~\ref{thm:main} asserts that $q$ is a Quillen
  bifibration as soon as \eqref{hyp:hBC} and
  \eqref{hyp:weak-conservative} are satisfied. The conditions of the
  lemma are precisely there to ensure that this is the case.
\end{proof}

Now recall that $\standsimp 1$ and $\standsimp 2$ are the posetal
categories associated to $\{0<1\}$ and $\{0<1<2\}$ respectively, and
write $c \from \standsimp 1 \to \standsimp 2$ for the functor
associated with the mapping $0 \mapsto 0, 1\mapsto 2$. Given a Reedy
category $\cat R$ and an object $r$ of degree $\mu$, denote
$\deginf i r \from \deginf{\cat R}\mu \to \deginf{\cat R}r$ the
inclusion of the full subcategory of $\cat R$ spanned by the object of
degree strictly less than $\mu$ into the one spanned by the same
objects plus $r$. Then proposition~\ref{prop:key-feature-reedy}
asserts that the following is a strict pullback square of categories:
\begin{displaymath}
  \begin{tikzcd}
    \functorcat{\deginf{\cat R}r}{\cat M} \rar \dar["\restr{\deginf i
      r}"{swap}] & \functorcat{\standsimp 2}{\cat M} \dar["\restr c"]
    \\
    \functorcat{\deginf{\cat R}\mu}{\cat M} \rar["\alpha_r"{swap}] &
    \functorcat{\standsimp 1}{\cat M}
  \end{tikzcd}
\end{displaymath}
where the bottom functor maps every diagram
$X \from \deginf {\cat R} \mu \to \cat M$ to the canonical arrow
$\alpha_r \from \latching r X \to \matching r X$. Moreover the functor
$\restr c$ is a Grothendieck bifibration: one can easily verify that
an arrow in $\functorcat{\standsimp 2}{\cat M}$
\begin{displaymath}
  \begin{tikzcd}
    \bullet \rar["f"] \dar & \bullet \dar
    \\
    \bullet \rar["g"] \dar & \bullet \dar
    \\
    \bullet \rar["h"] & \bullet
  \end{tikzcd}
\end{displaymath}
is cartesian if and only if the bottom square is a pullback, and is
cocartesian if and only if the top square is a pushout. In particular,
for each object $k\from A\to B$ of $\functorcat{\standsimp 1}{\cat M}$
we have a model structure on its fiber
$\fiber{(\restr c)}k \simeq \coslice{\slice{\cat M}B}A$. Stability of
cofibrations by pushout and of fibrations by pullback in the model
category $\cat M$ translates to say that
hypothesis~\ref{hyp:quillen-push-pull} is satisfied by $\restr c$. In
other word, by equipping the basis category
$\functorcat{\standsimp 1}{\cat M}$ with the trivial model structure,
theorem~\ref{thm:main} applies (\eqref{hyp:hBC} and
\eqref{hyp:weak-conservative} are vacuously met) and makes $\restr c$
a Quillen bifibration. The content of the proof of
proposition~\ref{prop:reedy-from-main-thm} is precisely showing
conditions~\ref{item:weak-conservative-through-F}
and~\ref{item:hBC-through-F} of
lemma~\ref{lem:quillen-bifib-by-pullback}. We can then conclude that
$\restr i r \from \functorcat{\deginf {\cat R} r}{\cat M} \to
\functorcat{\deginf {\cat R} r}{\cat M}$ is a Quillen bifibration as
in proposition~\ref{prop:reedy-from-main-thm}.

The result of \cite[Theorem 3.11]{shulman:gen-reedy} fall within this
view. Shulman defines the \define{bigluing} of a natural
transformation $\alpha \from F \to G$ between two functors
$F,G \from \cat M \to \cat N$ as the category $\bigluing \alpha$
whose:
\begin{itemize}
\item objects are factorizations
  \begin{displaymath}
    \alpha_M \from FM \overset f \to N \overset g \to GM
  \end{displaymath}
\item morphisms $(f,g) \overset {(h,k)} \to (f',g')$ are commutative
  diagrams of the form
  \begin{displaymath}
    \begin{tikzcd}
      FM \rar["f"] \dar["F(h)" swap] & N \rar["g"] \dar["k"] & GM \dar["G(h)"]
      \\
      FM' \rar["{f'}" swap] & N' \rar["{g'}" swap] & GM'
    \end{tikzcd}
  \end{displaymath}
\end{itemize}
Otherwise put, the category $\bigluing \alpha$ is a pullback as in:
\begin{displaymath}
  \begin{tikzcd}
    \bigluing \alpha \rar \dar & \functorcat{\standsimp 2}{\cat N}
    \dar["\restr c"]
    \\
    \cat M \rar["\alpha"{swap}] & \functorcat{\standsimp 1}{\cat N}
  \end{tikzcd}
\end{displaymath}
In the same fashion as in the proof of proposition
\ref{prop:reedy-from-main-thm}, we can show that
conditions~\ref{item:weak-conservative-through-F}
and~\ref{item:hBC-through-F} are satisfied for the bottom functor
(that we named abusively $\alpha$) when $F$ maps acyclic cofibrations
to {\em couniversal weak equivalences} and $G$ maps acyclic fibrations
to {\em universal weak equivalences}. By a couniversal weak
equivalence is meant a map every pushout of which is a weak
equivalence; and by a universal weak equivalence is meant a map every
pullback of which is a weak equivalence. Now
lemma~\ref{lem:quillen-bifib-by-pullback} directly proves Shulman's
theorem.
\begin{theorem}[Shulman]
  \label{thm:shulman}%
  Suppose $\cat N$ and $\cat M$ are both model categories. Let
  $\alpha \from F\to G$ between $F,G \from \cat M \to \cat N$
  satisfying that:
  \begin{itemize}
  \item $F$ is cocontinuous and maps acyclic cofibrations to
    couniversal weak equivalences,
  \item $G$ is continuous and maps acyclic fibrations to universal
    weak equivalence.
  \end{itemize}
  Then $\bigluing \alpha$ is a model category whose:
  \begin{itemize}
  \item cofibrations are the maps $(h,k)$ such that both $h$ and the map
    $FM' \sqcup_{FM} N \to N'$ induced by $k$ are cofibrations in
    $\cat M$ and $\cat N$ respectively,
  \item fibrations are the maps $(h,k)$ such that both $h$ and the map
    $N \to GM \times_{GM'} N'$ induced by $k$ are fibrations in
    $\cat M$ and $\cat N$ respectively,
  \item weak equivalences are the maps $(h,k)$ where both $h$ and $k$
    are weak equivalences in $\cat M$ and $\cat N$ respectively.
  \end{itemize}
\end{theorem}

Maybe the best way to understand this theorem is to see it at
play. Recall that a generalized Reedy category in the sense of
Berger and Moerdijk is a kind of Reedy category with degree preserving
isomorphism: precisely it is a category $\cat R$ with a degree
function $d: \ob{\cat R} \to \lambda$ and wide subcategories
$\cat R^+$ and $\cat R^-$ such that:
\begin{itemize}
\item \showcase{non-invertible} morphisms of $\cat R^+$ strictly raise
  the degree while those of $\cat R^-$ striclty lower it,
\item isomorphisms all preserve the degree,
\item $\cat R^+ \cap \cat R^-$ contains exactly the isomorphisms as
  morphisms,
\item every morphism $f$ can be factorized as $f=f^+f^-$ with
  $f^+ \in \cat R^+$ and $f^- \in\cat R^-$, and such a factorization
  is unique up to isomorphism,
\item if $\theta$ is an isomorphism and $\theta f = f$ for some
  $f \in \cat R^-$, then $\theta$ is an identity.
\end{itemize}
The central result in \cite{berger-moerdijk:gen-reedy} goes as follow:
\begin{enumerate}[label=(\arabic*)]
\item the latching and matching objects at $r\in\cat R$ of some
  $X \from \cat R \to \cat M$ are defined as in the classical case, but
  now the automorphism group $\aut r$ acts on them, so that
  $\latching r X$ and $\matching r X$ are objects of
  $\functorcat{\aut r}{\cat M}$ rather than mere objects of $\cat M$.
\item suppose $\cat M$ such that every $\functorcat{\aut r}{\cat M}$
  bears the projective model structure, and define Reedy cofibrations,
  Reedy fibrations and Reedy weak equivalences as usual but
  considering the usual induced maps
  $X_r \sqcup_{\latching r X} \latching r Y \to Y_r$ and
  $X_r \to Y_r \times_{\matching r Y} \matching r X$ in
  $\functorcat{\aut r}{\cat M}$, not in $\cat M$.
\item then Reedy cofibrations, Reedy fibrations and Reedy weak
  equivalences give $\functorcat{\cat R}{\cat M}$ a model structure.
\end{enumerate}
In that framework, theorem~\ref{thm:shulman} is applied repeatedly with
$\alpha$ being the canonical natural transformation between
$\latching r, \matching r \from \functorcat{\deginf{\cat R}\mu}{\cat
  M} \to \functorcat{\aut r}{\cat M}$ whenever $r$ is of degree
$\mu$. In particular, here we see the importance to be able to vary
the codomain category $\cat N$ of Shulman's result in each successor
step, and not to work with an homogeneous $\cat N$ all along.



\subsection{ Related works on Quillen bifibrations}
\label{subsec:versus-hp-rs}

Our work builds on the papers \cite{roig:model-bifibred},
\cite{stanculescu:bifib-model} on the one hand, and
\cite{harpaz-prasma:grothendieck-model} on the other hand, whose
results can be seen as special instances of our main theorem
\ref{thm:main}.  In these two lines of work, a number of sufficient
conditions are given in order to construct a Quillen bifibration.  The
fact that their conditions and constructions are special cases of ours
follows from the equivalence established in theorem \ref{thm:main}.
As a matter of fact, it is quite instructive to review and to point
out the divergences between the two approaches and ours, since it also
provides a way to appreciate the subtle aspects of our construction.

\medskip

Let us state the two results and comment them.
\begin{theorem}[Roig, Stanculescu]
  \label{thm:rs}%
  Let $p \from \cat E \to \cat B$ be a Grothendieck
  bifibration. Suppose that $\cat B$ is a model category with
  structure $(\totalcof\null,\totalweak\null,\totalfib\null)$ and that
  each fiber $\fiber {\cat E} A$ also with structure
  $(\totalcof A,\totalweak A,\totalfib A)$. Suppose also assumption
  \eqref{hyp:quillen-push-pull}. Then $\cat E$ is a model category with
  \begin{itemize}
  \item cofibrations the total ones,
  \item weak equivalences those $f \from X \to Y$ such that
    $p(f) \in \totalweak\null$ and $\pullfact f \in \totalweak{pX}$,
  \item fibrations the total ones,
  \end{itemize}
  provided that
  \begin{enumerate}[label=(\roman*)]
  \item\label{hyp:rs-hcons} $\pull u$ is homotopically conservative
    for all $u \in \totalweak\null$,
  \item\label{hyp:rs-unit} for $u \from A \to B$ an acyclic
    cofibration in $\cat B$, the unit of the adjoint pair
    $(\push u,\pull u)$ is pointwise a weak equivalence in
    $\fiber{\cat E}A$.
  \end{enumerate}
\end{theorem}
The formulation of the theorem is not symmetric, since it emphasizes
the cartesian morphisms over the cocartesian ones in the definition of
weak equivalences. This lack of symmetry in the definition of the weak
equivalences has the unfortunate effect of giving a similar bias to the
sufficient conditions: in order to obtain the weak factorization
systems, cocartesian morphisms above acyclic cofibrations should be
acyclic, which is the meaning of this apparently weird condition
\ref{hyp:rs-unit}; at the same time, cartesian morphisms above acyclic
fibrations should also be acyclic but this is vacuously true with the
definition of weak equivalences in theorem \ref{thm:rs}. Condition
\ref{hyp:rs-hcons} is only here for the 2-out-of-3 property, which
boils down to it.

\begin{theorem}[Harpaz, Prasma]
  Let $p \from \cat E \to \cat B$ be a Grothendieck
  bifibration. Suppose that $\cat B$ is a model category with
  structure $(\totalcof\null,\totalweak\null,\totalfib\null)$ and that
  each fiber $\fiber {\cat E} A$ also with structure
  $(\totalcof A,\totalweak A,\totalfib A)$. Suppose also assumption
  \eqref{hyp:quillen-push-pull}. Then $\cat E$ is a model category
  with
  \begin{itemize}
  \item cofibrations the total ones,
  \item weak equivalences those $f \from X \to Y$ such that
    $u = p(f) \in \totalweak\null$ and
    $\pull u(r) \circ \pullfact f \in \totalweak{pX}$, where
    $r \from Y \to Y^{\rm fib}$ is a fibrant replacement of $Y$ in
    $\fiber{\cat E}{pY}$,
  \item fibrations the total ones,
  \end{itemize}
  provided that
  \begin{enumerate}[label=(\roman*')]
  \item\label{hyp:hp-qequiv} the adjoint pair $(\push u,\pull u)$ is a
    Quillen equivalence for all $u \in \totalweak\null$,
  \item\label{hyp:hp-hcons} $\push u$ and $\pull v$ preserves weak equivalences whenever
    $u$ is an acyclic cofibration and $v$ an acyclic fibration.
  \end{enumerate}  
\end{theorem}
At first glance, Harpaz and Prasma introduces the same asymmetry that
Roig and Stanculescu in the definition of weak equivalences. They show
however that, under condition \ref{hyp:hp-qequiv}, weak equivalences
can be equivalently described as those $f \from X \to Y$ such
that $u = p(f) \in\totalweak\null$ and
\begin{displaymath}
  \push u X^{\rm cof} \to \push u X \to Y \in \totalweak{pY}
\end{displaymath}
where the first arrow is the image by $\push u$ of a cofibrant
replacement $X^{\rm cof} \to X$.  Hence, they manage to adapt
Roig-Stanculescu's result and to make it self dual. There is a cost
however, namely condition \ref{hyp:hp-qequiv}. Informally, it says
that weakly equivalent objects of $\cat B$ should have fibers with the
same homotopy theory. Harpaz and Prasma observe moreover that under
\ref{hyp:hp-qequiv}, \ref{hyp:rs-hcons} and \ref{hyp:rs-unit} implies
\ref{hyp:hp-hcons}. The condition is quite strong: in particular for
the simple Grothendieck bifibration
$\cod \from \functorcat{\lincat 2}{\cat B} \to \cat B$ of example
\ref{ex:main-thm-motiv}, it is equivalent to the fact that the model
category $\cat B$ is right proper. This explains why condition
\ref{hyp:hp-qequiv} has to be weakened in order to recover the Reedy
construction, as we do in this paper.

\medskip

It is possible to understand our work as a reflection on these
results, in the following way.
A common pattern in the train of thoughts developped in the three
papers
\cite{roig:model-bifibred,stanculescu:bifib-model,harpaz-prasma:grothendieck-model}
is their strong focus on cartesian and cocartesian morphisms above
\emph{weak equivalences}.  Looking at what it takes to construct weak
factorization systems using Stancuslescu's lemma (cf.\ lemma
\ref{lem:stan-lemma}), it is quite unavoidable to {\em push} along
(acyclic) cofibrations and {\em pull} along (acyclic) fibrations in
order to put everything in a common fiber, and then to use the
fiberwise model structure.  On the other hand, \emph{nothing} compels
us apparently to push or to pull along weak equivalences of $\cat B$
in order to define a model structure on $\cat E$.  This is precisely
the Ariadne's thread which we followed in the paper: organize
everything so that cocartesian morphisms above (acyclic) cofibrations
are (acyclic) cofibrations, and cartesian morphisms above (acyclic)
fibrations are (acyclic) fibrations.
This line of thought requires in particular to see every weak
equivalences of the basis category $\Bcategory$ as the
\emph{composite} of an acyclic cofibration followed by an acyclic
fibration.
One hidden source of inspiration for this divide comes from the
dualities of proof theory, and the intuition that pushing along an
(acyclic) cofibration should be seen as a \emph{positive operation}
(or a constructor) while pulling along an (acyclic) fibration should
be seen as a \emph{negative operation} (or a deconstructor),
see~\cite{mellies-zeilberger-lics-2016,mellies-zeilberger-mscs-2017}
for details.
All the rest, and in particular hypothesis
\eqref{hyp:weak-conservative} and \eqref{hyp:hBC}, follows from that
perspective, together with the idea of applying the framework to
reunderstand the Reedy construction from a bifibrational point of
view.

\medskip

Let us finally mention that we are currently preparing a companion
paper~\cite{cagne-mellies-companion-2017} where we carefully analyze
the relationship between the functor
$\Ho p: \Ho{\cat E} \to \Ho{\cat B}$ between the homotopy categories
$\Ho{\cat E}$ and $\Ho{\cat B}$ obtained from a Quillen bifibration
$p:\cat E\to\cat B$ by Quillen localisation, and the Grothendieck
bifibration $q:\cat F\to\cat B$ obtained by localising each fiber
$\fiber{\cat E}A$ of the Quillen bifibration~$p$ independently as
$\fiber{\cat F} A = \Ho{\fiber{\cat E}A}$.

\addcontentsline{toc}{section}{References}
\bibliographystyle{alpha}
\bibliography{main}

\end{document}

%% file: macros.tex
\DeclareMathOperator{\oboperator}{Ob}
\newcommand{\ob}[1]{\oboperator{#1}}
\newcommand{\op}[1]{{#1}^{\circ}}

\newcommand{\slice}[2]{#1_{\kern-2pt/\kern-.5pt#2}}
\newcommand{\coslice}[2]{\vphantom{#1}_{#2\kern-.5pt\backslash\kern-2pt}#1}

\newcommand{\id}[1]{\boldsymbol{1}_{#1}}

\newcommand{\fiber}[2]{{#1}_{#2}}
\newcommand{\fiberprod}[3]{{#1}\times_{#3}{#2}}

\DeclareMathOperator*{\limoperator}{lim}
\renewcommand{\lim}[2][\null]{\limoperator_{#1}\left( #2 \right)}

\newcommand{\limend}[1]{\int_{#1}}
\newcommand{\colimend}[1]{\int^{#1}}
\newcommand{\tens}{\mathbin \odot}
\newcommand{\cotens}{\mathbin \pitchfork}

\newcommand{\adjointarrows}{\rightleftarrows}

\DeclareMathOperator{\homoperator}{Hom}
\renewcommand{\hom}[3][\null]{%
        \ifx#1\null%
                \homoperator\left(#2,#3\right)%
        \else #1\left(#2,#3\right) \fi}
\DeclareMathOperator{\Homoperator}{\underline{Hom}}

\newcommand{\functorcat}[3][\null]{\ifx#1\null[{#2},{#3}]
                                    \else\Homoperator(#2,#3) \fi}

\newcommand{\dom}{\mathrm{dom}}
\newcommand{\cod}{\mathrm{cod}}

\newcommand{\cart}[3][\null]{\rho^{#1}_{#2,#3}}
\newcommand{\cartz}{\rho}

\newcommand{\cocart}[3][\null]{\lambda^{#1}_{#2,#3}}
\newcommand{\cocartz}{\lambda}
\newcommand{\push}[2]{{#1}_!#2}
\newcommand{\pull}[2]{{#1}^\ast#2}
\newcommand{\pushfact}[1]{{#1}_{\triangleright}}
\newcommand{\pullfact}[1]{{#1}^{\triangleleft}}
\newcommand{\middlefact}[3]{\vphantom{#1}_{#2}{#1}^{#3}}
\newcommand{\bigluing}[1]{\operatorname{\mathcal G\kern -.1em\ell}\left(#1\right)}

\newcommand{\restr}[1]{{#1}^\ast}

\newcommand{\cat}[1]{\mathscr{#1}}
\newcommand{\concrete}[1]{\mathsf{#1}}
\newcommand{\Set}{\concrete{Set}}

\newcommand{\Cat}{\concrete{Cat}}
\newcommand{\Adj}{\concrete{Adj}}

\newcommand{\lincat}[1]{\boldsymbol{#1}}
\newcommand{\standsimp}[1]{\Delta[#1]}

\newcommand{\Ho}[2][\null]{\mathbf{Ho}_{\rm #1}\left( #2 \right)}
\newcommand{\class}[1]{\mathfrak{#1}}

\newcommand{\weakorth}{\mathbin{%
    \tikz[scale=.2,anchor=base,baseline]{%
      \draw (0,0) -- (1,0) -- (1,1) -- (0,1) -- (0,0) -- (1,1);%
    }%
  }%
}

\newcommand{\deginf}[2]{#1_{#2}}
\DeclareMathOperator{\latchingoperator}{L}

\newcommand{\latching}[1]{\latchingoperator_{#1}}
\DeclareMathOperator{\matchingoperator}{M}

\newcommand{\matching}[1]{\matchingoperator_{#1}}

\DeclareMathOperator{\autoperator}{Aut}
\newcommand{\aut}[1]{\autoperator\left(#1\right)}

\newcommand*\from{:}

\newcommand*\cofrom{:}

\newcommand{\define}[1]{\emph{#1}}
\newcommand{\showcase}[1]{\textbf{#1}}

\newtheorem{theorem}{Theorem}[section]
\newtheorem*{theorem*}{Theorem}
\newtheorem{proposition}[theorem]{Proposition}
\newtheorem{lemma}[theorem]{Lemma}
\newtheorem{corollary}[theorem]{Corollary}

\newtheorem*{claim*}{Claim}
\theoremstyle{definition}
\newtheorem{definition}[theorem]{Definition}

\theoremstyle{remark}
\newtheorem{remark}[theorem]{\sc Remark}

\newtheorem{example}[theorem]{\sc Example(s)}